\documentclass[12pt]{amsart}
\usepackage{enumerate}
\usepackage{hyperref}
\usepackage{graphicx,color}
\usepackage{cite}
\usepackage{cancel}
\usepackage{amsthm, amsmath, amssymb, mathtools}
\usepackage[top=45truemm, bottom=45truemm, left=30truemm, right=30truemm]{geometry}

\renewcommand{\epsilon}{\varepsilon}
\renewcommand{\limsup}{\varlimsup}
\renewcommand{\liminf}{\varliminf}

\newcommand{\abs}[1]{\left|#1\right|}

\newcounter{count}
\newcommand{\num}{\stepcounter{count}\the\value{count}}

\numberwithin{equation}{section}

\newtheorem{theorem}{Theorem}[section]
\newtheorem{lemma}[theorem]{Lemma}
\newtheorem{corollary}[theorem]{Corollary}

\newtheorem{conjecture}[theorem]{Conjecture}

\theoremstyle{definition}

\newcommand{\cI}{\mathcal{I}}

\newcommand{\cN}{\mathcal{N}}

\newcommand{\cR}{\mathcal{R}}

\newcommand{\PS}{\mathrm{PS}}
\newcommand{\AP}{\mathrm{AP}}

\title[Asymptotic and non-asymptotic results for a binary additive problem ...]{Asymptotic and non-asymptotic results \\ for a binary additive problem \\ involving Piatetski-Shapiro numbers}

\author[Y.\ Yoshida]{Yuuya Yoshida}
\address{Yuuya Yoshida\\
Nagoya Institute of Technology\\ Gokiso-cho\\ Showa-ku\\ Nagoya\\ 466-8555\\ Japan}
\curraddr{}
\email{yyoshida9130@gmail.com}



\begin{document}
\maketitle

\begin{abstract}
For all $\alpha_1,\alpha_2\in(1,2)$ with $1/\alpha_1+1/\alpha_2>5/3$, 
we show that the number of pairs $(n_1,n_2)$ of positive integers with $N=\lfloor{n_1^{\alpha_1}}\rfloor+\lfloor{n_2^{\alpha_2}}\rfloor$ 
is equal to $\Gamma(1+1/\alpha_1)\Gamma(1+1/\alpha_2)\Gamma(1/\alpha_1+1/\alpha_2)^{-1}N^{1/\alpha_1+1/\alpha_2-1} + o(N^{1/\alpha_1+1/\alpha_2-1})$ 
as $N\to\infty$, 
where $\Gamma$ denotes the gamma function.
Moreover, we show a non-asymptotic result for the same counting problem 
when $\alpha_1,\alpha_2\in(1,2)$ lie in a larger range than the above.
Finally, we give some asymptotic formulas for similar counting problems in a heuristic way.
\end{abstract}

\section{Introduction}\label{intro}

Let $\PS(\alpha)$ be the set $\{ \lfloor{n^\alpha}\rfloor : n\in\mathbb{N} \}$ for $\alpha\ge1$, 
where $\lfloor{x}\rfloor$ (resp.\ $\lceil{x}\rceil$) denotes 
the greatest (resp.\ least) integer $\le x$ (resp.\ $\ge x$) for a real number $x$, 
and $\mathbb{N}$ denotes the set of all positive integers.
The set $\PS(\alpha)$ is called a \textit{Piatetski-Shapiro sequence} when $\alpha>1$ is non-integral.

Segal \cite{Segal1, Segal2} showed that, 
for every $\alpha>1$, there exists a positive integer $s_0(\alpha)$ satisfying the following condition: 
if $s\ge s_0(\alpha)$, then every sufficiently large integer $N\ge1$ has a representation of the form 
\begin{equation}
	N = \lfloor{n_1^\alpha}\rfloor+\cdots+\lfloor{n_s^\alpha}\rfloor
	\label{eqWaring}
\end{equation}
with integers $n_1,\ldots,n_s\ge0$.
This statement can be regarded as a fractional version of the Waring problem, 
since the case $\alpha\in\mathbb{N}$ is just the original Waring problem.
The minimum number of such $s_0(\alpha)$ is often denoted by $G(\alpha)$, 
which many researcher has studied \cite{Segal1,Segal2,Deshouillers1,Arkhipov,Listratenko, Deshouillers2,Gritsenko,Konyagin}.
For large $s$ depending on an exponent $\alpha$, the asymptotic formula for the number of representations is known \cite{Deshouillers1,Arkhipov}.
For example, Arkhipov and Zhitkov \cite{Arkhipov} showed that, 
for every $\alpha>12$ and every integer $s>22\alpha^2(\log\alpha+4)$, 
the number of representations of the form \eqref{eqWaring} with integers $n_1,\ldots,n_s\ge1$ is 
\begin{equation}
	\frac{\Gamma(1+1/\alpha)^s}{\Gamma(s/\alpha)}N^{s/\alpha-1} + o(N^{s/\alpha-1})
	\quad(N\to\infty). \label{eqWaring2}
\end{equation}
However, one cannot take small $s$ such as $s=2$ in this statement.

The case $s=2$ has been studied mainly on the existence of a representation \cite{Deshouillers2,Gritsenko,Konyagin}.
Deshouillers \cite{Deshouillers2} showed that, for every $\alpha\in(1,4/3)$ and every sufficiently large integer $N\ge1$, 
there exists a pair $(n_1,n_2)$ of positive integers such that 
\begin{equation}
	N = \lfloor{n_1^\alpha}\rfloor+\lfloor{n_2^\alpha}\rfloor.
	\label{eqWaring3}
\end{equation}
Gritsenko \cite{Gritsenko} and Konyagin \cite{Konyagin} improved the above range $(1,4/3)$ to $(1,55/41)$ and $(1,3/2)$, respectively.
One might expect that the asymptotic formula for the number of such pairs $(n_1,n_2)$ is obtained by formally substituting $s=2$ into \eqref{eqWaring2}, 
but we were not able to find any references showing it rigorously.

In this paper, we show asymptotic and non-asymptotic results for the number of the above pairs $(n_1,n_2)$.
A way of proving them, except for estimating error terms, is simple and useful to solve any other counting problems involving Piatetski-Shapiro numbers, 
especially, in solving them heuristically (see Section~\ref{heuristic}).
For real numbers $\alpha_1,\alpha_2>1$ and an integer $N\ge1$, define the number $\cR_{\alpha_1,\alpha_2}(N)$ as 
\[
\cR_{\alpha_1,\alpha_2}(N) = \#\{ (n_1,n_2)\in\mathbb{N}^2 : \lfloor{n_1^{\alpha_1}}\rfloor+\lfloor{n_2^{\alpha_2}}\rfloor=N \}.
\]
When $\alpha_1=\alpha_2=\alpha$, we write $\cR_\alpha(N)$ instead of $\cR_{\alpha_1,\alpha_2}(N)$.
Although we state Theorems~\ref{main1}, \ref{main2}, and \ref{mainDes2} below, 
they are proved in Section~\ref{proof} and Appendix.

\subsection{Asymptotic results}

First, we state asymptotic results.

\begin{theorem}\label{main1}
	Let $\alpha_1,\alpha_2\in(1,2)$ and $1/\alpha_1+1/\alpha_2>5/3$.
	Then 
	\[
	\lim_{N\to\infty} \frac{\cR_{\alpha_1,\alpha_2}(N)}{N^{1/\alpha_1+1/\alpha_2-1}}
	= \alpha_1^{-1}\alpha_2^{-1}B(1/\alpha_1,1/\alpha_2)
	= \frac{\Gamma(1+1/\alpha_1)\Gamma(1+1/\alpha_2)}{\Gamma(1/\alpha_1+1/\alpha_2)}.
	\]
	where $B$ and $\Gamma$ denote the beta function and the gamma function, respectively.
\end{theorem}

The following corollary follows from Theorem~\ref{main1} immediately.

\begin{corollary}\label{main1'}
	Let $\alpha\in(1,6/5)$. Then 
	\[
	\lim_{N\to\infty} \frac{\cR_\alpha(N)}{N^{2/\alpha-1}}
	= \frac{\Gamma(1+1/\alpha)^2}{\Gamma(2/\alpha)}.
	\]
\end{corollary}

The asymptotic formula in Corollary~\ref{main1'} is the same as that obtained by formally substituting $s=2$ into \eqref{eqWaring2}.
Also, by considering the sum $\sum_{1\le n<x} \cR_\alpha(\lfloor{n^\alpha}\rfloor)$, 
we obtain the following corollary.

\begin{corollary}\label{main1''}
	Let $\alpha\in(1,6/5)$. Then 
	\begin{equation}
		\lim_{x\to\infty} \frac{\#\{ (l,m,n)\in\mathbb{N}^3 : n<x,\ 
		\lfloor{l^\alpha}\rfloor+\lfloor{m^\alpha}\rfloor=\lfloor{n^\alpha}\rfloor \}}{x^{3-\alpha}}
		= \frac{\Gamma(1+1/\alpha)^2}{(3-\alpha)\Gamma(2/\alpha)}.
		\label{eq03}
	\end{equation}
\end{corollary}

The asymptotic formula in Corollary~\ref{main1''} is probably true if $\alpha\in(1,2)\cup(2,3)$.
We check this conjecture in a heuristic way in Section~\ref{heuristic}.
However, the case $\alpha=2$ does not hold because it is known \cite{Stronina} that 
\begin{equation}
\begin{split}
	&\quad \#\{ (l,m,n)\in\mathbb{N}^3 : n<x,\ l^2+m^2=n^2 \}\\
	&= \frac{1}{\pi}x\log x + Bx + O\bigl( x^{1/2}\exp(-C(\log x)^{3/5}(\log\log x)^{-1/5}) \bigr) \quad (x\to\infty)
\end{split}\label{eqN3}
\end{equation}
for an explicit constant $B$ and some $C>0$.
Hence, the case $\alpha=2$ is a kind of singularity.
We can observe this situation by numerical computation (see Section~\ref{heuristic}).

\subsection{Non-asymptotic results}

Next, we state non-asymptotic results.
In this case, we can improve the assumption $1/\alpha_1+1/\alpha_2 > 5/3$ of Theorem~\ref{main1}.

\begin{theorem}\label{main2}
	Let $\alpha_1,\alpha_2\in(1,2)$ and $1/\alpha_1+1/\alpha_2>3/2$.
	Then, for every integer $N\ge1$, 
	\[
	\cR_{\alpha_1,\alpha_2}(N) \ll N^{1/\alpha_1+1/\alpha_2-1},
	\]
	where the implicit constant is absolute.
\end{theorem}

Theorem~\ref{main2} is still true even if replacing the assumption $1/\alpha_1+1/\alpha_2>3/2$ and the inequality $\ll$ 
with $1/\alpha_1+1/\alpha_2 \ge 3/2$ and $\ll_{\alpha_1,\alpha_2}$, respectively.
However, we show only the case $1/\alpha_1+1/\alpha_2>3/2$ in this paper for simplicity.

For a lower bound, we can use and generalize Deshouillers' proof, 
since he \cite{Deshouillers2} showed that $\cR_\alpha(N) \gg N^{2/\alpha-1}$ as $N\to\infty$ for every $\alpha\in(1,4/3)$.
From now on, denote by $\alpha_{\max}$ (resp.\ $\alpha_{\min}$) the maximum (resp.\ minimum) number of $\alpha_1$ and $\alpha_2$.

\begin{theorem}\label{mainDes2}
	Let $\alpha_1,\alpha_2\in(1,3/2)$, $1/\alpha_1+1/\alpha_2>3/2$, and $6/\alpha_{\max}+1/\alpha_{\min}>13/3$.
	Then 
	\[
	\cR_{\alpha_1,\alpha_2}(N) \gg N^{1/\alpha_1+1/\alpha_2-1}
	\quad(N\to\infty),
	\]
	where the implicit constant is absolute.
\end{theorem}

Theorem~\ref{mainDes2} can be proved in a similar way to Deshouillers' proof, 
but we should note the additional assumptions $\alpha_1,\alpha_2<3/2$ and $6/\alpha_{\max}+1/\alpha_{\min}>13/3$.
One can ignore these assumptions when $\alpha_1=\alpha_2=\alpha$ (because the inequality $2/\alpha>3/2$ implies that $\alpha<3/2$ and $7/\alpha>13/3$).
We prove Theorem~\ref{mainDes2} in Appendix.

The following corollary follows from Theorems~\ref{main2} and \ref{mainDes2} immediately.

\begin{corollary}\label{main2'}
	Let $\alpha\in(1,4/3)$.
	Then 
	\[
	\cR_\alpha(N) \asymp N^{2/\alpha-1} \quad (N\to\infty),
	\]
	where the implicit constant is absolute.
\end{corollary}

By considering the sum $\sum_{1\le n<x} \cR_\alpha(\lfloor{n^\alpha}\rfloor)$, 
we obtain the following corollary.

\begin{corollary}\label{main2''}
	Let $\alpha\in(1,4/3)$. Then 
	\[
	\#\{ (l,m,n)\in\mathbb{N}^3 : n<x,\ \lfloor{l^\alpha}\rfloor+\lfloor{m^\alpha}\rfloor=\lfloor{n^\alpha}\rfloor \}
	\asymp x^{3-\alpha} \quad (x\to\infty),
	\]
	where the implicit constant is absolute.
\end{corollary}

Moreover, we can estimate the number of (non-trivial) three-term arithmetic progressions (for short, $3$-APs) in $\PS(\alpha)$ as follows.

\begin{corollary}\label{main2'''}
	Let $\alpha\in(1,4/3)$. Then 
	\begin{equation}
		\#\{ (l,m,n)\in\mathbb{N}^3 : l<m<n<x,\ \lfloor{l^\alpha}\rfloor+\lfloor{n^\alpha}\rfloor=2\lfloor{m^\alpha}\rfloor \}
		\asymp x^{3-\alpha} \quad (x\to\infty),
		\label{eq01}
	\end{equation}
	where the implicit constant is absolute.
\end{corollary}

Actually, the following statement related to Corollary~\ref{main2'''} is known \cite[Theorem~1.3]{SY2}: 
for every $\alpha\in(1,2)$ and every integer $k\ge3$, 
the number of $k$-term arithmetic progressions $P\subset\mathbb{N}\cap(0,x)$ such that 
$(\lfloor{n^\alpha}\rfloor)_{n\in P}$ is also a $k$-term arithmetic progression, 
is $\asymp_{\alpha,k} x^{2-\alpha/2}$ as $x\to\infty$.
Since $3-\alpha>2-\alpha/2$ for every $\alpha\in(1,2)$, 
the number of such $3$-APs is far less than the number of the other $3$-APs in $\PS(\alpha)$
when $\alpha\in(1,4/3)$.

A heuristic argument supports that the non-asymptotic formula in Corollary~\ref{main2'''} is still true even if $\alpha\in(1,2)\cup(2,3)$.
However, the case $\alpha=2$ does not hold because 
Hulse et al.\ \cite[Theorem~8.1]{Hulse} showed the asymptotic formula for the number of primitive $3$-APs in squares, 
which yields that 
\begin{equation}
\begin{split}
	&\quad \#\{ (l,m,n)\in\mathbb{N}^3 : l<m<n<x,\ l^2+n^2=2m^2 \}\\
	&= \frac{\sqrt{2}\log(1+\sqrt{2})}{\pi^2}x\log x + O(x) \quad (x\to\infty).
\end{split}\label{eqNAP}
\end{equation}
Hence, the case $\alpha=2$ is a kind of singularity as well as the case $\alpha=2$ of \eqref{eq03}.
We can observe this situation by numerical computation (see Section~\ref{heuristic}).

\begin{proof}[Proof of Corollary~$\ref{main2'''}$ assuming Corollary~$\ref{main2'}$]
	Corollary~\ref{main2'} implies that 
	\begin{equation}
		\#\{ (l,m,n)\in\mathbb{N}^3 : m<x,\ \lfloor{l^\alpha}\rfloor+\lfloor{n^\alpha}\rfloor=2\lfloor{m^\alpha}\rfloor \}
		\asymp x^{3-\alpha} \quad (x\to\infty)
		\label{eq02}
	\end{equation}
	by considering the sum $\sum_{1\le m<x} \cR_\alpha(2\lfloor{m^\alpha}\rfloor)$.
	Also, we have the following inequalities: 
	\begin{align*}
		&\quad \#\{ (l,m,n)\in\mathbb{N}^3 : l<m<n,\ m<2^{-1/\alpha}x,\ \lfloor{l^\alpha}\rfloor+\lfloor{n^\alpha}\rfloor=2\lfloor{m^\alpha}\rfloor \}\\
		&\le \#\{ (l,m,n)\in\mathbb{N}^3 : l<m<n<x,\ \lfloor{l^\alpha}\rfloor+\lfloor{n^\alpha}\rfloor=2\lfloor{m^\alpha}\rfloor \}\\
		&\le \#\{ (l,m,n)\in\mathbb{N}^3 : m<x,\ \lfloor{l^\alpha}\rfloor+\lfloor{n^\alpha}\rfloor=2\lfloor{m^\alpha}\rfloor \}
	\end{align*}
	and 
	\begin{align*}
		&\quad \#\{ (l,m,n)\in\mathbb{N}^3 : l<m<n,\ m<2^{-1/\alpha}x,\ \lfloor{l^\alpha}\rfloor+\lfloor{n^\alpha}\rfloor=2\lfloor{m^\alpha}\rfloor \}\\
		&= \frac{1}{2}\bigl( \#\{ (l,m,n)\in\mathbb{N}^3 : m<2^{-1/\alpha}x,\ \lfloor{l^\alpha}\rfloor+\lfloor{n^\alpha}\rfloor=2\lfloor{m^\alpha}\rfloor \}\\
		&\qquad\qquad- \#\{ (l,m,n)\in\mathbb{N}^3 : l=m=n<2^{-1/\alpha}x,\ \lfloor{l^\alpha}\rfloor+\lfloor{n^\alpha}\rfloor=2\lfloor{m^\alpha}\rfloor \} \bigr)\\
		&\ge \frac{1}{2}\#\{ (l,m,n)\in\mathbb{N}^3 : m<2^{-1/\alpha}x,\ \lfloor{l^\alpha}\rfloor+\lfloor{n^\alpha}\rfloor=2\lfloor{m^\alpha}\rfloor \}\\
		&\quad- 2^{-1/\alpha}x.
	\end{align*}
	By these and \eqref{eq02}, we obtain \eqref{eq01}.
\end{proof}

\subsection{Related work}

First, we mention two existing studies that address the counting problem for the equation 
\begin{equation}
	N = \lfloor{n_1^{\alpha_1}+n_2^{\alpha_2}}\rfloor
	\label{eqWaring3'}
\end{equation}
instead of \eqref{eqWaring3}.
Churchhouse \cite{Churchhouse} proved that, for every $\alpha_1=\alpha_2=\alpha\in(1,4/3)$, 
the number of pairs $(n_1,n_2)$ of positive integers with \eqref{eqWaring3'} is $\gg_\alpha N^{2/\alpha-1}$ as $N\to\infty$.
This is probably the first attack intending to solve the counting problem for \eqref{eqWaring3} (since \eqref{eqWaring3} was mentioned in \cite{Churchhouse}).
Also, Rieger \cite{Rieger} proved that, for every $\alpha_1,\alpha_2\in(1,2)$ with $1/\alpha_1+1/\alpha_2>3/2$, 
the number of pairs $(n_1,n_2)$ of positive integers with \eqref{eqWaring3'} is $\gg_{\alpha_1,\alpha_2} N^{1/\alpha_1+1/\alpha_2-1}$ as $N\to\infty$.
In this statement, one does not need to assume $6/\alpha_{\max}+1/\alpha_{\min}>13/3$, which appears in Theorem~\ref{mainDes2}.

Next, we state some existing studies on binary additive problems involving Piatetski-Shapiro numbers.
For the case of more than two variables, see other references, e.g., \cite{Akbal1,Akbal2,Dimitrov}.
As already stated, the case $s=2$ of \eqref{eqWaring} were studied in \cite{Deshouillers2,Gritsenko,Konyagin}.
Many other existing studies investigate the case when at least one variable $n_1$ is a prime number \cite{Kumchev,Yu, Petrov,Wu}.
Kumchev \cite{Kumchev} proved that, for every $\alpha\in(1,16/15)$ and every sufficiently large integer $N\ge1$, 
there exist a prime $p$ and an integer $m\ge1$ such that 
\begin{equation}
	N = \lfloor{p^\alpha}\rfloor+\lfloor{m^\alpha}\rfloor.
	\label{eq04}
\end{equation}
Yu \cite{Yu} improved the above range $(1,16/15)$ to $(1,11/10)$.
Petrov and Tolev \cite{Petrov} proved that, for every $\alpha\in(1,29/28)$ and every sufficiently large integer $N\ge1$, 
there exist a prime $p$ and an almost prime $m\ge1$ with at most $\lfloor{52/(29-28\alpha)}\rfloor+1$ prime factors such that \eqref{eq04} holds.
Recently, Wu \cite{Wu} improved Petrov and Tolev's result by replacing $(1,29/28)$ and $\lfloor{52/(29-28\alpha)}\rfloor+1$ with 
$(1,247/238)$ and $\lfloor{450/(247-238\alpha)}\rfloor+1$, respectively.

Some existing studies address a similar problem with almost all $N$ \cite{Balanzario,Yu, Laporta,Zhu}, 
where ``almost all'' means for the set of exceptions to have density zero.
Let $x\ge2$.
Balanzario et al.\ \cite{Balanzario} proved that, for every $\alpha\in(1,17/11)$ and almost all integers $N\ge1$, 
there exist a prime $p$ and an integer $m\ge1$ such that \eqref{eq04} holds, 
where the number of exceptional integers $N\in[1,x]$ is $O(x^{1-\delta})$ for some $\delta=\delta(\alpha)>0$.
Yu \cite{Yu} improved the above range $(1,17/11)$ to $\bigl( 1,(1+\sqrt{5})/2 \bigr)$.
Laporta \cite{Laporta} proved that, for every $\alpha\in(1,17/16)$ and almost all integers $N\in(x/2,x]$, 
there exists a pair $(p_1,p_2)$ of primes such that 
\begin{equation}
	N = \lfloor{p_1^\alpha}\rfloor+\lfloor{p_2^\alpha}\rfloor,
	\label{eq05}
\end{equation}
where the number of exceptional integers $N\in(x/2,x]$ is $O\bigl( x\exp(-(\log x)^{1/6}) \bigr)$.
Zhu \cite{Zhu} improved the above range $(1,17/16)$ to $(1,14/11)$.
Moreover, Laporta \cite{Laporta} and Zhu \cite{Zhu} also showed the asymptotic formula for the weighted count of such pairs $(p_1,p_2)$: 
\[
\sum_{\substack{p_1,p_2\text{ primes} \\ \text{with \eqref{eq05}}}} (\log p_1)(\log p_2)
= \frac{\Gamma(1+1/\alpha)^2}{\Gamma(2/\alpha)}N^{2/\alpha-1} + O\bigl( x^{2/\alpha-1}\exp(-(\log x)^{1/6}) \bigr)
\]
for every $\alpha$ in their ranges and almost all integers $N\in(x/2,x]$. 

Finally, we remark an existing study related to Corollary~\ref{main1'}.
Using differences of two Piatetski-Shapiro numbers instead of sums of them, 
the author \cite{Yoshida} proved the following asymptotic formula: 
for every $\alpha\in\bigl( 1,(\sqrt{21}+4)/5 \bigr)$, 
the number of pairs $(m,n)$ of positive integers with $d=\lfloor{n^\alpha}\rfloor - \lfloor{m^\alpha}\rfloor$ 
is equal to 
\[
\beta\alpha^{-\beta}\zeta(\beta)d^{\beta-1} + o(d^{\beta-1}) \quad (d\to\infty),
\]
where $\beta=1/(\alpha-1)$, and $\zeta$ denotes the Riemann zeta function.
Since $(\sqrt{21}+4)/5$ is about $1.717$, 
the range $\bigl( 1,(\sqrt{21}+4)/5 \bigr)$ is much larger than that of Corollary~\ref{main1'}.

\section{Preliminary lemmas}\label{pre-lem}

From now on, denote by $\{x\}=x-\lfloor{x}\rfloor$ the fractional part of a real number $x$, and 
by $\|\cdot\|$ the distance to the nearest integer.
We use the notations ``$o(\cdot)$, $O(\cdot)$, $\sim$, $\asymp$, $\ll$'' in the usual sense.
If implicit constants depend on parameters $a_1,\ldots,a_n$, 
we often write ``$O_{a_1,\ldots,a_n}(\cdot)$, $\asymp_{a_1,\ldots,a_n}$, $\ll_{a_1,\ldots,a_n}$'' 
instead of ``$O(\cdot)$, $\asymp$, $\ll$''.
Also, denote by $e(x)$ the function $e^{2\pi ix}$, and 
by $|\cI|$ the length of an interval $\cI$ of $\mathbb{R}$.

We begin with the following lemma which is useful to solve counting problems involving Piatetski-Shapiro numbers.

\begin{lemma}[Koksma \cite{Koksma}; cf.\ {\cite[Proposition~2]{Deshouillers2}}]\label{Koksma}
	Let $\cI$ be an interval of $\mathbb{R}$.
	For $i=1,2$, let $f_i$ be a real-valued function defined on $\cI$, and $a_i,b_i$ be real numbers with $0\le a_i<b_i<1$.
	Then, for all $H_1,H_2\ge4$, the value 
	\begin{align*}
		R &= \#\bigl\{ n\in\cI\cap\mathbb{Z} : a_1\le\{f_1(n)\}<b_1,\ a_2\le\{f_2(n)\}<b_2 \bigr\}\\
		&\quad- \#(\cI\cap\mathbb{Z})(b_1-a_1)(b_2-a_2)
	\end{align*}
	satisfies the inequality 
	\begin{align*}
		R &\ll \sum_{\substack{1\le\abs{h_1}\le H_1 \\ 1\le\abs{h_2}\le H_2}}
		\Bigl( \prod_{i=1,2} \min\{ \|b_i-a_i\|+H_i^{-1}, \abs{h_i}^{-1} \} \Bigr)\abs{\sum_{n\in\cI\cap\mathbb{Z}} e\bigl( h_1f_1(n)+h_2f_2(n) \bigr)}\\
		&\quad+ (b_2-a_2+H_2^{-1})\sum_{1\le h\le H_1} \min\{ \|b_1-a_1\|+H_1^{-1}, h^{-1} \}\abs{\sum_{n\in\cI\cap\mathbb{Z}} e\bigl( hf_1(n) \bigr)}\\
		&\quad+ (b_1-a_1+H_1^{-1})\sum_{1\le h\le H_2} \min\{ \|b_2-a_2\|+H_2^{-1}, h^{-1} \}\abs{\sum_{n\in\cI\cap\mathbb{Z}} e\bigl( hf_2(n) \bigr)}\\
		&\quad+ \#(\cI\cap\mathbb{Z})\Bigl( \frac{b_1-a_1}{H_2} + \frac{b_2-a_2}{H_1} + \frac{1}{H_1H_2} \Bigr).
	\end{align*}
\end{lemma}

By Lemma~\ref{Koksma}, a counting problem reduces to estimating exponential sums.
The following lemmas are useful to estimate exponential sums.

\begin{lemma}[Kusmin--Landau; cf.\ {\cite[Theorem~2.1]{GK}}]\label{1stderiv}
	Let $\cI$ be an interval of $\mathbb{R}$, and 
	$f\colon \cI\to\mathbb{R}$ be a $C^1$ function such that $f'$ is monotone.
	If $\lambda_1>0$ satisfies that 
	\[
	\|f'(x)\| \ge \lambda_1
	\] 
	for all $x\in\cI$, then 
	\[
	\sum_{n\in\cI\cap\mathbb{Z}} e(f(n)) \ll \lambda_1^{-1}. 
	\]
\end{lemma}

\begin{lemma}[van der Corput; cf.\ {\cite[Theorem~2.2]{GK}}]\label{2ndderiv}
	Let $\cI$ be an interval of $\mathbb{R}$ with $|\cI|\ge1$, 
	$f\colon \cI\to\mathbb{R}$ be a $C^2$ function, and 
	$c\ge1$ be a real number.
	If $\Lambda_2\ge\lambda_2>0$ satisfies that 
	\[
	\lambda_2 \le |f''(x)| \le \Lambda_2
	\]
	for all $x\in\cI$, then 
	\[
	\sum_{n\in\cI\cap\mathbb{Z}} e(f(n))
	\ll \abs{\cI}\Lambda_2\lambda_2^{-1/2} + \lambda_2^{-1/2}.
	\]
	In particular, if $\lambda_2>0$ satisfies that 
	\[
	\lambda_2 \le |f''(x)| \le c\lambda_2
	\] 
	for all $x\in\cI$, then 
	\[
	\sum_{n\in\cI\cap\mathbb{Z}} e(f(n))
	\ll_c \abs{\cI}\lambda_2^{1/2} + \lambda_2^{-1/2}.
	\]
\end{lemma}

\begin{lemma}[Sargos \cite{Sargos} and Gritsenko \cite{Grisenko}]\label{3rdderiv}
	Let $\cI$ be an interval of $\mathbb{R}$ with $|\cI|\ge1$, 
	$f\colon \cI\to\mathbb{R}$ be a $C^3$ function, and 
	$c\ge1$ be a real number.
	If $\lambda_3>0$ satisfies that 
	\[
	\lambda_3 \le |f'''(x)| \le c\lambda_3
	\]
	for all $x\in\cI$, then 
	\[
	\sum_{n\in\cI\cap\mathbb{Z}} e(f(n))
	\ll_c \abs{\cI}\lambda_3^{1/6} + \lambda_3^{-1/3}.
	\]
\end{lemma}

To estimate an exponential sum, we sometimes use an \textit{exponent pair}, 
which can be applied to a function ``well-approximated'' by a model phase function defined below.
However, we omit the definition of an exponent pair and 
only state a fact used in this paper, 
since we apply an exponent pair only to a model phase function.

\begin{lemma}[Exponent pair; cf.\ {\cite[Chapter~3]{GK}}]\label{exp-pair}
	Let $N$ be a positive integer, and $\cI$ be an interval contained in the interval $[N,2N]$.
	For $y,s>0$, define the model phase function $\varphi=\varphi_{y,s}$ as 
	\[
	\varphi(x) =
	\begin{dcases}
		\frac{yx^{1-s}}{1-s} & s\not=1,\\
		y\log x & s=1.
	\end{dcases}
	\]
	If $(k,l)$ is an exponent pair, then 
	\[
	\sum_{n\in\cI\cap\mathbb{Z}} e(\varphi(n))
	\ll_{s,k,l} L^kN^l + L^{-1},
	\]
	where $L \coloneqq yN^{-s} = \varphi'(N)$.
	For example, the pairs $(0,1)$, $(1/2,1/2)$, and $(1/6,2/3)$ are exponent pairs.
\end{lemma}

Next, we prove several lemmas on real numbers, which are related to the ranges of $1/\alpha_1$ and $1/\alpha_2$ in Theorems~\ref{main1} and \ref{main2}.

\begin{lemma}\label{lem01}
	Let $x_1,x_2\in(1/2,1)$ and $x_1+x_2>3/2$.
	Then the following inequalities hold: 
	\renewcommand{\theenumi}{$\arabic{enumi}$}
	\renewcommand{\labelenumi}{\theenumi.}
	\begin{enumerate}
		\item
		$x_1+2(x_2-1)(2-x_1-x_2) > 0$;
		\item
		$x_2+2(x_1-1)(2-x_1-x_2) > 0$;
		\item
		$(2-x_2)(2-x_1-x_2) < x_1$;
		\item
		$(2-x_1)(2-x_1-x_2) < x_2$.
	\end{enumerate}
\end{lemma}
\begin{proof}
	Let us show inequality~1.
	By $2-x_1-x_2>0$ and $x_2-1>-1/2$, we have 
	\begin{align*}
		x_1+2(x_2-1)(2-x_1-x_2)
		&> x_1 - (2-x_1-x_2)\\
		&= 2x_1+x_2-2.
	\end{align*}
	Since the assumptions $x_1\in(1/2,1)$ and $x_1+x_2>3/2$ imply that 
	\begin{equation}
		2x_1+x_2-2 > 1/2+3/2-2 = 0, \label{eq06}
	\end{equation}
	inequality~1 holds.
	Inequality~2 can also be proved in the same way.
	\par
	By $2-x_1>0$, $2-x_1-x_2<1/2$, and \eqref{eq06}, we obtain inequality~3: 
	\[
	(2-x_2)(2-x_1-x_2) < 1-x_2/2 < x_1.
	\]
	Inequality~4 can also be proved in the same way.
\end{proof}

\begin{lemma}\label{lem01'}
	Let $x_1,x_2\in(1/2,1)$ and $x_1+x_2>3/2$.
	For $i=1,2$, set the real number $X_i$ as $X_i=(1-x_i)/(2-x_1-x_2)$.
	Then there exist real numbers $\beta_1$ and $\beta_2$ satisfying the following inequalities: 
	\begin{equation}
	\begin{split}
		&\quad \max\Bigl\{ 0,\ 2(1-x_2) - \frac{1}{2-x_1-x_2} + 2X_1 \Bigr\}\\
		&< \beta_1
		< \min\Bigl\{ X_1,\ 2(1-x_2),\ \frac{1}{2-x_1-x_2}-2 \Bigr\}
	\end{split}\label{eq:beta1}
	\end{equation}
	and 
	\begin{equation}
	\begin{split}
		&\quad \max\Bigl\{ 0,\ 2(1-x_1) - \frac{1}{2-x_1-x_2} + 2X_2 \Bigr\}\\
		&< \beta_2 < \min\Bigl\{ X_2,\ 2(1-x_1),\ \frac{1}{2-x_1-x_2}-2 \Bigr\}.
	\end{split}\label{eq:beta2}
	\end{equation}
	Moreover, for each $i=1,2$, the above $\beta_i$ satisfies that 
	\begin{equation}
		\beta_i < \min\Bigl\{ \frac{1}{2} + \frac{1/2}{2-x_1-x_2} - X_{3-i},\ \frac{1}{3-x_{3-i}} \Bigr\}.
		\label{eq:beta12}
	\end{equation}
\end{lemma}
\begin{proof}
	We show that the desired $\beta_1$ exists.
	It suffices to show the following inequalities: 
	\renewcommand{\theenumi}{$\arabic{enumi}$}
	\renewcommand{\labelenumi}{\theenumi.}
	\begin{enumerate}
		\item
		$X_1>0$, $2(1-x_2)>0$, $1/(2-x_1-x_2)-2>0$;
		\item
		$2(1-x_2) - 1/(2-x_1-x_2) + 2X_1 < X_1$;
		\item
		$2(1-x_2) - 1/(2-x_1-x_2) + 2X_1 < 2(1-x_2)$;.
		\item
		$2(1-x_2) - 1/(2-x_1-x_2) + 2X_1 < 1/(2-x_1-x_2)-2$.
	\end{enumerate}
	Inequality~1 is trivial due to $x_1,x_2\in(1/2,1)$ and $x_1+x_2>3/2$.
	Since the equivalences 
	\begin{align*}
		&\quad 2(1-x_2) - \frac{1}{2-x_1-x_2} + 2X_1 < X_1\\
		&\iff 2(1-x_2) - \frac{1}{2-x_1-x_2} + X_1 < 0\\
		&\iff 2(1-x_2)(2-x_1-x_2) - 1 + (1-x_1) < 0\\
		&\iff 2(x_2-1)(2-x_1-x_2) + x_1 > 0
	\end{align*}
	hold, Lemma~\ref{lem01} implies inequality~2.
	Also, the equivalences 
	\begin{align*}
		&\quad 2(1-x_2) - \frac{1}{2-x_1-x_2} + 2X_1 < 2(1-x_2)\\
		&\iff 2X_1 < \frac{1}{2-x_1-x_2}
		\iff 2(1-x_1) < 1
		\iff x_1>1/2
	\end{align*}
	imply inequality~3.
	Inequality~4 follows from Lemma~\ref{lem01} and the equivalences below: 
	\begin{align*}
		&\quad 2(1-x_2) - \frac{1}{2-x_1-x_2} + 2X_1 < \frac{1}{2-x_1-x_2}-2\\
		&\iff 2(1-x_2) - \frac{2x_1}{2-x_1-x_2} < -2\\
		&\iff (2-x_2)(2-x_1-x_2) < x_1.
	\end{align*}
	Therefore, the desired $\beta_1$ exists.
	The existence of the desired $\beta_2$ can also be proved in the same way.
	\par
	We show \eqref{eq:beta12}.
	By $x_2\in(1/2,1)$ and \eqref{eq:beta1}, 
	the inequality $\beta_1<1/(3-x_2)$ holds: 
	\begin{align*}
		(3-x_2)\beta_1
		&= \beta_1 + \beta_1 + (1-x_2)\beta_1\\
		&< X_1 + 2(1-x_2) + (1-x_2)\Bigl( \frac{1}{2-x_1-x_2}-2 \Bigr)\\
		&= X_1+X_2 = 1.
	\end{align*}
	Similarly, the inequality $\beta_2<1/(3-x_1)$ holds.
	Thus, the remainder of \eqref{eq:beta12} is the inequality $\beta_i < 1/2 + 1/2(2-x_1-x_2) - X_{3-i}$.
	By \eqref{eq:beta1} and \eqref{eq:beta2}, it suffices to show that $X_i < 1/2 + 1/2(2-x_1-x_2) - X_{3-i}$.
	This follows from $x_1+x_2>3/2$ and the equivalences below: 
	\begin{align*}
		&\quad X_i < \frac{1}{2} + \frac{1/2}{2-x_1-x_2} - X_{3-i}\\
		&\iff 1 = X_1+X_2 < \frac{1}{2} + \frac{1/2}{2-x_1-x_2}\\
		&\iff 2-x_1-x_2 < 1
		\iff 1 < x_1+x_2.
	\end{align*}
\end{proof}

\begin{lemma}\label{lem01''}
	Let $x_1,x_2\in(1/2,1)$ and $x_1+x_2 > 3/2$.
	Let $\beta_1$ and $\beta_2$ be real numbers with \eqref{eq:beta1} and \eqref{eq:beta2}.
	Then there exists a real number $2-x_1-x_2<\hat{\gamma}_0<1$ satisfying the following inequalities: 
	\renewcommand{\theenumi}{$\arabic{enumi}$}
	\renewcommand{\labelenumi}{\theenumi.}
	\begin{enumerate}
		\item
		$\hat{\gamma}_0\beta_1 < 1-x_1$;
		\item
		$-\hat{\gamma}_0(\beta_1/2+x_2-1) + x_2-1/2 < x_1+x_2-1$;
		\item
		$\hat{\gamma}_0\beta_1/2+1/2 < x_1+x_2-1$;
		\item
		$\hat{\gamma}_0(\beta_1-1/2)+x_1-1/2 < x_1+x_2-1$;
		\item
		$\hat{\gamma}_0(\beta_1+1-x_2)+x_2-1 < x_1+x_2-1$;
		\item
		$\hat{\gamma}_0\beta_2 < 1-x_2$;
		\item
		$-\hat{\gamma}_0(\beta_2/2+x_1-1) + x_1-1/2 < x_1+x_2-1$;
		\item
		$\hat{\gamma}_0\beta_2/2+1/2 < x_1+x_2-1$;
		\item
		$\hat{\gamma}_0(\beta_2-1/2)+x_2-1/2 < x_1+x_2-1$;
		\item
		$\hat{\gamma}_0(\beta_2+1-x_1)+x_1-1 < x_1+x_2-1$.
	\end{enumerate}
\end{lemma}
\begin{proof}
	First, inequality~5 follows from inequalities~1--3: 
	\begin{align*}
		\hat{\gamma}_0(\beta_1+1-x_2)
		&= \hat{\gamma}_0\beta_1 + \hat{\gamma}_0\beta_1/2 - \hat{\gamma}_0(\beta_1/2+x_2-1)\\
		&< (1-x_1) + (x_1+x_2-3/2) + (x_1-1/2)\\
		&= x_1+x_2-1 < x_1.
	\end{align*}
	Similarly, inequality~10 follows from inequality~6--8.
	\par
	Now, the inequality $2-x_1-x_2<1$ is clear.
	If substituting $\hat{\gamma}_0=2-x_1-x_2$ into inequalities~1--4 and 6--9, 
	they turn to the following inequalities: 
	\renewcommand{\theenumi}{$\arabic{enumi}'$}
	\renewcommand{\labelenumi}{\theenumi.}
	\begin{enumerate}
		\item
		$\beta_1 < X_1$;
		\item
		$-(\beta_1/2+x_2-1) + X_1 - 1/2(2-x_1-x_2) < 0$;
		\item
		$\beta_1/2 < 1/2(2-x_1-x_2)-1$;
		\item
		$\beta_1-1/2-1/2(2-x_1-x_2) < -X_2$;
		\item
		$\beta_2 < X_2$;
		\item
		$-(\beta_2/2+x_1-1) + X_2 - 1/2(2-x_1-x_2) < 0$;
		\item
		$\beta_2/2 < 1/2(2-x_1-x_2)-1$;
		\item
		$\beta_2-1/2-1/2(2-x_1-x_2) < -X_1$,
	\end{enumerate}
	where $X_1$ and $X_2$ are defined in Lemma~\ref{lem01'}.
	If inequalities~$1'$--$4'$ and $6'$--$9'$ hold, 
	then the desired $\hat{\gamma}_0$ exists.
	Thus, it suffices to show inequalities~$1'$--$4'$ and $6'$--$9'$, 
	which hold by \eqref{eq:beta1}, \eqref{eq:beta2}, and \eqref{eq:beta12}.
	Therefore, the desired $\hat{\gamma}_0$ exists.
\end{proof}

\section{Proofs of Theorems~$\ref{main1}$ and $\ref{main2}$}\label{proof}

In this section, we prove Theorems~\ref{main1} and \ref{main2}.
For $\alpha>1$, define the function $\phi_\alpha$ as 
\[
\phi_\alpha(x) = (x+1)^{1/\alpha} - x^{1/\alpha}.
\]
For the above purpose, we need to show several lemmas.

\begin{lemma}\label{lem02} 
	Let $\alpha_1,\alpha_2\in(1,2)$, $\beta_1,\beta_2\ge0$, $c_1,c_2,c_3,c_4\in(0,1)$, $c_1,c_2\in[1/N,1-1/N]$, and $c_2-c_1\ge1/N$.
	Set $\hat{c}_i=1-c_i$ for $i=1,2,3,4$.
	For every integer 
	\begin{equation}
		N \ge \max\{ c_1^{-1}(16\alpha_1^{-1}c_4^{-\beta_1})^{\alpha_1/(\alpha_1-1)},\ 
		\hat{c}_2^{-1}(16\alpha_2^{-1}c_3^{-\beta_2})^{\alpha_2/(\alpha_2-1)} \},
		\label{eqN}
	\end{equation}
	the value 
	\begin{align*}
		R &= \#\biggl\{ c_1N<n\le c_2N : 
		\begin{array}{c}
			\{-n^{1/\alpha_1}\} < \phi_{\alpha_1}(c_3N),\\
			\{-(N-n)^{1/\alpha_2}\} < \phi_{\alpha_2}(c_4N)
		\end{array}
		\biggr\}\\
		&\quad- (c_2-c_1)N\phi_{\alpha_1}(c_3N)\phi_{\alpha_2}(c_4N)
	\end{align*}
	satisfies the following inequalities: 
	\[
	R \ll R_0 + R_1 + R_2 + R_3,
	\]
	\begin{equation}
		R_0 \ll \abs{\{c_2N\}-\{c_1N\}}c_3^{1/\alpha_1-1}c_4^{1/\alpha_2-1}N^{1/\alpha_1+1/\alpha_2-2},
		\label{eqR0}
	\end{equation}
	\begin{equation}
	\begin{split}
		R_1 &\ll (c_2-c_1)( c_3^{1/\alpha_1-1-\beta_2}\hat{c}_2^{1/\alpha_2-1}+c_4^{1/\alpha_2-1-\beta_1}c_1^{1/\alpha_1-1}\\
		&\qquad\qquad\qquad\qquad+ c_4^{-\beta_1}c_3^{-\beta_2}c_1^{1/\alpha_1-1}\hat{c}_2^{1/\alpha_2-1} )N^{1/\alpha_1+1/\alpha_2-1},
	\end{split}\label{eqR1}
	\end{equation}
	\begin{equation}
	\begin{split}
		R_2 &\ll_{\alpha_1,\alpha_2} C_2c_2^{1/2\alpha_1}c_4^{\beta_1/2}c_1^{(1/2)(1-1/\alpha_1)}N^{1/\alpha_2-1/2}\\
		&\qquad+ C_1\hat{c}_1^{1/2\alpha_2}c_3^{\beta_2/2}\hat{c}_2^{(1/2)(1-1/\alpha_2)}N^{1/\alpha_1-1/2}\\
		&\qquad+ C_1C_2( c_2^{1-1/\alpha_1}N^{1/\alpha_2-1} + \hat{c}_1^{1-1/\alpha_2}N^{1/\alpha_1-1} )\log N,\quad\text{and}
	\end{split}\label{eqR2}
	\end{equation}
	\begin{equation}
	\begin{split}
		R_3 &\ll_{\alpha_1,\alpha_2} ( 1+C_1c_2^{1/2\alpha_1}c_4^{\beta_1/2}c_1^{(1/2)(1-1/\alpha_1)}\\
		&\qquad+ C_2\hat{c}_1^{1/2\alpha_2}c_3^{\beta_2/2}\hat{c}_2^{(1/2)(1-1/\alpha_2)} )N^{1/2}\log N\\
		&\qquad+ \min\{ C_1c_2^{1-1/\alpha_1}, C_2\hat{c}_1^{1-1/\alpha_2} \}(\log N)^2,
	\end{split}\label{eqR3}
	\end{equation}
	where $C_1=\max\{ c_3^{1/\alpha_1-1}, c_4^{-\beta_1}c_1^{1/\alpha_1-1} \}$ and 
	$C_2=\max\{ c_4^{1/\alpha_2-1}, c_3^{-\beta_2}\hat{c}_2^{1/\alpha_2-1} \}$.
\end{lemma}
\begin{proof}
	Let $N$ be an integer with \eqref{eqN}.
	Set the positive numbers $H_1$ and $H_2$ as 
	\begin{equation}
		H_1 = (\alpha_1/4)c_4^{\beta_1}(c_1N)^{1-1/\alpha_1},\quad
		H_2 = (\alpha_2/4)c_3^{\beta_2}(\hat{c}_2N)^{1-1/\alpha_2}.
		\label{eqH}
	\end{equation}
	By \eqref{eqN}, both $H_1$ and $H_2$ are greater than or equal to $4$.
	Noting that $\phi_{\alpha_i}(c_{i+2}N) \ll (c_{i+2}N)^{1/\alpha_i-1}$ for $i=1,2$, 
	we have 
	\begin{align*}
		\phi_{\alpha_1}(c_3N) + H_1^{-1}
		&\ll (c_3N)^{1/\alpha_1-1} + c_4^{-\beta_1}(c_1N)^{1/\alpha_1-1}
		\ll C_1N^{1/\alpha_1-1},\\
		\phi_{\alpha_1}(c_4N) + H_2^{-1}
		&\ll (c_4N)^{1/\alpha_2-1} + c_3^{-\beta_2}(\hat{c}_2N)^{1/\alpha_2-1}
		\ll C_2N^{1/\alpha_2-1}.
	\end{align*}
	These and Lemma~\ref{Koksma} imply that 
	\begin{align*}
		R &\ll \abs{(c_2-c_1)N - \#(\cI\cap\mathbb{Z})}\phi_{\alpha_1}(c_3N)\phi_{\alpha_2}(c_4N)\\
		&\quad+ \sum_{\substack{1\le\abs{h_1}\le H_1 \\ 1\le\abs{h_2}\le H_2}}
		\Bigl( \prod_{i=1,2} \min\{ C_iN^{1/\alpha_i-1}, \abs{h_i}^{-1} \} \Bigr)
		\abs{\sum_{n\in\cI\cap\mathbb{Z}} e\bigl( h_1f_1(n)+h_2f_2(n) \bigr)}\\
		&\quad+ C_2N^{1/\alpha_2-1}\sum_{1\le h\le H_1} \min\{ C_1N^{1/\alpha_1-1}, h^{-1} \}\abs{\sum_{n\in\cI\cap\mathbb{Z}} e\bigl( hf_1(n) \bigr)}\\
		&\quad+ C_1N^{1/\alpha_1-1}\sum_{1\le h\le H_2} \min\{ C_2N^{1/\alpha_2-1}, h^{-1} \}\abs{\sum_{n\in\cI\cap\mathbb{Z}} e\bigl( hf_2(n) \bigr)}\\
		&\quad+ \#(\cI\cap\mathbb{Z})\Bigl( \frac{(c_3N)^{1/\alpha_1-1}}{H_2} + \frac{(c_4N)^{1/\alpha_2-1}}{H_1} + \frac{1}{H_1H_2} \Bigr).
	\end{align*}
	where $f_1(x)=x^{1/\alpha_1}$, $f_2(x)=(N-x)^{1/\alpha_2}$, and $\cI=(c_1N, c_2N]$.
	Denote by $R_0$ (resp.\ $R_3$, $R_{2,1}$, $R_{2,2}$, $R_1$) 
	the first (resp.\ second, third, fourth, fifth) line of the right-hand side: 
	\[
	R \ll R_0 + R_3 + R_{2,1} + R_{2,2} + R_1.
	\]
	Also, partition the sum $R_3$ into two sums: 
	\begin{align*}
		R_3 = \sum_{\substack{1\le\abs{h_1}\le H_1 \\ 1\le\abs{h_2}\le H_2}}
		= \sum_{\substack{1\le\abs{h_1}\le H_1 \\ 1\le\abs{h_2}\le H_2 \\ h_1h_2>0}}
		+ \sum_{\substack{1\le\abs{h_1}\le H_1 \\ 1\le\abs{h_2}\le H_2 \\ h_1h_2<0}}
		= R_{3,1}+R_{3,2},\quad\text{say}.
	\end{align*}
	\par
	\setcounter{count}{0}
	\textbf{Step~\num.}
	Eq.~\eqref{eqR0} is clear.
	Eq.~\eqref{eqR1} follows from the inequality $\#(\cI\cap\mathbb{Z}) \le (c_2-c_1)N+1 \le 2(c_2-c_1)N$ and \eqref{eqH}.
	\par
	\textbf{Step~\num.}
	Using the exponent pair $(1/2,1/2)$ (see Lemma~\ref{exp-pair}), we have 
	\begin{align*}
		\sum_{n\in\cI\cap\mathbb{Z}} e\bigl( hf_1(n) \bigr)
		&\ll \sum_{j=\lfloor{\log_2(c_1N)}\rfloor}^{\lfloor{\log_2(c_2N)}\rfloor}
		\abs{\sum_{\max\{2^j,c_1N\}<n\le\min\{2^{j+1},c_2N\}} e\bigl( hf_1(n) \bigr)}\\
		&\ll_{\alpha_1} \sum_{j=\lfloor{\log_2(c_1N)}\rfloor}^{\lfloor{\log_2(c_2N)}\rfloor}
		\Bigl( \bigl( h(2^j)^{1/\alpha_1-1} \bigr)^{1/2}(2^j)^{1/2} + \bigl( h(2^j)^{1/\alpha_1-1} \bigr)^{-1} \Bigr)\\
		&= \sum_{j=\lfloor{\log_2(c_1N)}\rfloor}^{\lfloor{\log_2(c_2N)}\rfloor}
		\bigl( h^{1/2}(2^j)^{1/2\alpha_1} + h^{-1}(2^j)^{1-1/\alpha_1} \bigr)\\
		&\ll h^{1/2}(c_2N)^{1/2\alpha_1} + h^{-1}(c_2N)^{1-1/\alpha_1}
	\end{align*}
	for every $h\ge1$.
	This yields that 
	\begin{align*}
		R_{2,1} &\ll_{\alpha_1} C_2N^{1/\alpha_2-1}\sum_{1\le h\le H_1} \min\{ C_1N^{1/\alpha_1-1}, h^{-1} \}
		\bigl( h^{1/2}(c_2N)^{1/2\alpha_1} + h^{-1}(c_2N)^{1-1/\alpha_1} \bigr)\\
		&\le C_2N^{1/\alpha_2-1}\sum_{1\le h\le H_1}
		\bigl( h^{-1}\cdot h^{1/2}(c_2N)^{1/2\alpha_1} + C_1N^{1/\alpha_1-1}\cdot h^{-1}(c_2N)^{1-1/\alpha_1} \bigr)\\
		&\ll C_2c_2^{1/2\alpha_1}H_1^{1/2}N^{1/2\alpha_1+1/\alpha_2-1} + C_1C_2c_2^{1-1/\alpha_1}(\log H_1)N^{1/\alpha_2-1}\\
		&\ll C_2c_2^{1/2\alpha_1}c_4^{\beta_1/2}c_1^{(1/2)(1-1/\alpha_1)}N^{1/\alpha_2-1/2} + C_1C_2c_2^{1-1/\alpha_1}N^{1/\alpha_2-1}\log N,
	\end{align*}
	where we have used \eqref{eqH} to obtain the last inequality.
	Similarly, it follows that 
	\[
	\sum_{n\in\cI\cap\mathbb{Z}} e\bigl( hf_2(n) \bigr)
	\ll_{\alpha_2} h^{1/2}(\hat{c}_1N)^{1/2\alpha_2} + h^{-1}(\hat{c}_1N)^{1-1/\alpha_2}
	\]
	and 
	\[
	R_{2,2} \ll_{\alpha_2}
	C_1\hat{c}_1^{1/2\alpha_2}c_3^{\beta_2/2}\hat{c}_2^{(1/2)(1-1/\alpha_2)}N^{1/\alpha_1-1/2} + C_1C_2\hat{c}_1^{1-1/\alpha_2}N^{1/\alpha_1-1}\log N.
	\]
	Therefore, the value $R_2 \coloneqq R_{2,1}+R_{2,2}$ satisfies \eqref{eqR2}.
	\par
	\textbf{Step~\num.}
	Let $h_1h_2>0$. From now on, set $f=h_1f_1+h_2f_2$.
	When $x\in(0,N/2]\cap(2^j,2^{j+1}]$, we have 
	\[
	\abs{f''(x)} \asymp_{\alpha_1,\alpha_2} \abs{h_1}(2^j)^{1/\alpha_1-2}+\abs{h_2}N^{1/\alpha_2-2}.
	\]
	This and Lemma~\ref{2ndderiv} imply that 
	\begin{align*}
		&\quad \sum_{c_1N<n\le\min\{ c_2N, N/2 \}} e(f(n))\\
		&\ll \sum_{j=\lfloor{\log_2(c_1N)}\rfloor}^{\lfloor{\log_2(c_2N)}\rfloor} \abs{\sum_{\max\{2^j,c_1N\}<n\le\min\{2^{j+1}, c_2N, N/2\}} e(f(n))}\\
		&\ll_{\alpha_1,\alpha_2} \sum_{j=\lfloor{\log_2(c_1N)}\rfloor}^{\lfloor{\log_2(c_2N)}\rfloor}
		\Bigl( 2^j\bigl( \abs{h_1}^{1/2}(2^j)^{1/2\alpha_1-1}+\abs{h_2}^{1/2}N^{1/2\alpha_2-1} \bigr)\\
		&\qquad\qquad\qquad\qquad+ \min\{ \abs{h_1}^{-1/2}(2^j)^{1-1/2\alpha_1}, \abs{h_2}^{-1/2}N^{1-1/2\alpha_2} \} \Bigr)\\
		&\ll \abs{h_1}^{1/2}(c_2N)^{1/2\alpha_1}+\abs{h_2}^{1/2}c_2N^{1/2\alpha_2}
		+ \abs{h_1}^{-1/2}(c_2N)^{1-1/2\alpha_1}\\
		&\le c_2^{1/2\alpha_1}( \abs{h_1}^{1/2}N^{1/2\alpha_1}+\abs{h_2}^{1/2}N^{1/2\alpha_2}
		+ \abs{h_1}^{-1/2}N^{1-1/2\alpha_1} ),
	\end{align*}
	where we have used the inequality $\max\{ c_1, c_2^{1-1/2\alpha_1} \} \le c_2^{1/2\alpha_1}$ to obtain the last inequality.
	Similarly, 
	\begin{align*}
		&\quad \sum_{\max\{ c_1N, N/2 \}<n\le c_2N} e(f(n))\\
		&\ll_{\alpha_1,\alpha_2} \hat{c}_1^{1/2\alpha_2}( \abs{h_1}^{1/2}N^{1/2\alpha_1}+\abs{h_2}^{1/2}N^{1/2\alpha_2}
		+ \abs{h_2}^{-1/2}N^{1-1/2\alpha_2} ).
	\end{align*}
	Thus, 
	\begin{align*}
		\sum_{n\in\cI\cap\mathbb{Z}} e(f(n))
		&\ll_{\alpha_1,\alpha_2} (c_2^{1/2\alpha_1}+\hat{c}_1^{1/2\alpha_2})(\abs{h_1}^{1/2}N^{1/2\alpha_1} + \abs{h_2}^{1/2}N^{1/2\alpha_2})\\
		&\quad+ c_2^{1/2\alpha_1}\abs{h_1}^{-1/2}N^{1-1/2\alpha_1} + \hat{c}_1^{1/2\alpha_2}\abs{h_2}^{-1/2}N^{1-1/2\alpha_2}
	\end{align*}
	and 
	\begin{align*}
		R_{3,1} &\ll_{\alpha_1,\alpha_2} \sum_{\substack{1\le h_1\le H_1 \\ 1\le h_2\le H_2}}
		\Bigl( \prod_{i=1,2} \min\{ C_iN^{1/\alpha_i-1}, \abs{h_i}^{-1} \} \Bigr)\\
		&\qquad\qquad\Bigl( (c_2^{1/2\alpha_1}+\hat{c}_1^{1/2\alpha_2})(\abs{h_1}^{1/2}N^{1/2\alpha_1} + \abs{h_2}^{1/2}N^{1/2\alpha_2})\\
		&\qquad\qquad+ c_2^{1/2\alpha_1}\abs{h_1}^{-1/2}N^{1-1/2\alpha_1} + \hat{c}_1^{1/2\alpha_2}\abs{h_2}^{-1/2}N^{1-1/2\alpha_2} \Bigr).
	\end{align*}
	By the inequalities $c_2^{1/2\alpha_1}+\hat{c}_1^{1/2\alpha_2}\le2$ and 
	\begin{align*}
		\prod_{i=1,2} \min\{ C_iN^{1/\alpha_i-1}, \abs{h_i}^{-1} \}
		\le \abs{h_1}^{-1}\abs{h_2}^{-1},\ C_1N^{1/\alpha_1-1}\abs{h_2}^{-1},\ \abs{h_1}^{-1}C_2N^{1/\alpha_2-1},
	\end{align*}
	it turns out that 
	\begin{align*}
		R_{3,1} &\ll_{\alpha_1,\alpha_2} \sum_{\substack{1\le h_1\le H_1 \\ 1\le h_2\le H_2}}
		\Bigl( h_1^{-1}h_2^{-1}( \abs{h_1}^{1/2}N^{1/2\alpha_1} + \abs{h_2}^{1/2}N^{1/2\alpha_2} )\\
		&\qquad\qquad
		+ C_1N^{1/\alpha_1-1}h_2^{-1}\cdot c_2^{1/2\alpha_1}\abs{h_1}^{-1/2}N^{1-1/2\alpha_1}\\
		&\qquad\qquad
		+ h_1^{-1}C_2N^{1/\alpha_2-1}\cdot\hat{c}_1^{1/2\alpha_2}\abs{h_2}^{-1/2}N^{1-1/2\alpha_2} \Bigr)\\
		&\ll H_1^{1/2}(\log H_2)N^{1/2\alpha_1}+(\log H_1)H_2^{1/2}N^{1/2\alpha_2}\\
		&\quad+ C_1c_2^{1/2\alpha_1}H_1^{1/2}(\log H_2)N^{1/2\alpha_1}
		+ C_2\hat{c}_1^{1/2\alpha_2}(\log H_1)H_2^{1/2}N^{1/2\alpha_2}\\
		&\ll_{\alpha_1,\alpha_2} ( 1+C_1c_2^{1/2\alpha_1}c_4^{\beta_1/2}c_1^{(1/2)(1-1/\alpha_1)}\\
		&\qquad\qquad+ C_2\hat{c}_1^{1/2\alpha_2}c_3^{\beta_2/2}\hat{c}_2^{(1/2)(1-1/\alpha_2)} )N^{1/2}\log N,
	\end{align*}
	where we have used \eqref{eqH} to obtain the last inequality.
	\par
	\textbf{Step~\num.}
	Let $h_1h_2<0$, $\abs{h_1}\le H_1$, and $\abs{h_2}\le H_2$.
	When $x\in\cI$, by \eqref{eqH} we have 
	\begin{align*}
		\abs{f'(x)} &= \abs{h_1}\alpha_1^{-1}x^{1/\alpha_1-1}+\abs{h_2}\alpha_2^{-1}(N-x)^{1/\alpha_2-1}\\
		&\le \abs{h_1}\alpha_1^{-1}(c_1N)^{1/\alpha_1-1}+\abs{h_2}\alpha_2^{-1}(\hat{c}_2N)^{1/\alpha_2-1}
		\le 1/2
	\end{align*}
	and 
	\begin{equation*}
		\|f'(x)\| = \abs{f'(x)} \gg \abs{h_1}(c_2N)^{1/\alpha_1-1}+\abs{h_2}(\hat{c}_1N)^{1/\alpha_2-1}.
	\end{equation*}
	Since $f'''(x)$ has a constant sign, 
	$f''$ is strictly monotone.
	Moreover, $f''(+0)=\pm\infty$ and $f''(N-0)=\mp\infty$, where double-sign corresponds.
	By these facts, there exists a unique $x_0\in(0,N)$ such that $f''(x_0)=0$.
	Thus, $f'$ is monotone in the intervals $(0,x_0]$ and $[x_0,N)$.
	Lemma~\ref{1stderiv} implies that 
	\begin{align*}
		\sum_{n\in\cI\cap\mathbb{Z}} e(f(n))
		&\ll \min\{ \abs{h_1}^{-1}(c_2N)^{1-1/\alpha_1}, \abs{h_2}^{-1}(\hat{c}_1N)^{1-1/\alpha_2} \}\\
		&\le \abs{h_1}^{-1}(c_2N)^{1-1/\alpha_1}.
	\end{align*}
	This yields that 
	\begin{align*}
		R_{3,2} &\ll \sum_{\substack{1\le h_1\le H_1 \\ 1\le h_2\le H_2}}
		\Bigl( \prod_{i=1,2} \min\{ C_iN^{1/\alpha_i-1}, \abs{h_i}^{-1} \} \Bigr)
		\abs{h_1}^{-1}(c_2N)^{1-1/\alpha_1}\\
		&\le \sum_{\substack{1\le h_1\le H_1 \\ 1\le h_2\le H_2}}
		C_1N^{1/\alpha_1-1}\abs{h_2}^{-1}\cdot\abs{h_1}^{-1}(c_2N)^{1-1/\alpha_1}\\
		&\ll C_1c_2^{1-1/\alpha_1}(\log H_1)(\log H_2)
		\ll_{\alpha_1,\alpha_2} C_1c_2^{1-1/\alpha_1}(\log N)^2,
	\end{align*}
	where we have used \eqref{eqH} to obtain the last inequality.
	Similarly, 
	\[
	R_{3,2} \ll_{\alpha_1,\alpha_2} C_2\hat{c}_1^{1-1/\alpha_2}(\log N)^2.
	\]
	Therefore, 
	\[
	R_{3,2} \ll_{\alpha_1,\alpha_2} \min\{ C_1c_2^{1-1/\alpha_1}, C_2\hat{c}_1^{1-1/\alpha_2} \}(\log N)^2.
	\]
	\par
	\textbf{Step~\num.}
	By steps~3--4, the value $R_3=R_{3,1}+R_{3,2}$ satisfies \eqref{eqR3}.
\end{proof}

\begin{lemma}\label{lem02'} 
	Let $\alpha_1,\alpha_2\in(1,2)$, $\epsilon\in(0,1/2)$, $c_1,c_2,c_3,c_4\in[\epsilon,1-\epsilon]$, $c_1,c_2\in[1/N,1-1/N]$, and $c_2-c_1\ge1/N$.
	For every $\gamma>0$ and every integer $N\ge\exp(4^{1/\gamma})$, 
	the value 
	\begin{align*}
		R &= \#\biggl\{ c_1N<n\le c_2N : 
		\begin{array}{c}
			\{-n^{1/\alpha_1}\} < \phi_{\alpha_1}(c_3N),\\
			\{-(N-n)^{1/\alpha_2}\} < \phi_{\alpha_2}(c_4N)
		\end{array}
		\biggr\}\\
		&\quad- (c_2-c_1)N\phi_{\alpha_1}(c_3N)\phi_{\alpha_2}(c_4N)
	\end{align*}
	satisfies the following inequalities: 
	\[
	R \ll_\epsilon R_1 + R_2 + R_3,
	\]
	\begin{equation}
		R_1 \ll_{\epsilon,\gamma} N^{1/\alpha_1+1/\alpha_2-1}(\log N)^{-\gamma},
		\label{eqR1'}
	\end{equation}
	\begin{equation}
		R_2 \ll_{\alpha_1,\alpha_2,\gamma} N^{1/\alpha_{\min}-1/2}(\log N)^{\gamma/2}, \quad\text{and}
		\label{eqR2'}
	\end{equation}
	\begin{equation}
		R_3 \ll_{\alpha_1,\alpha_2} N^{2/3}(\log N)^{2\gamma/3+1}.
		\label{eqR3'}
	\end{equation}
\end{lemma}
\begin{proof}
	Let $\gamma>0$ and $N\ge\exp(4^{1/\gamma})$.
	Set the positive numbers $H_1$ and $H_2$ as 
	\begin{equation}
		H_1 = N^{1-1/\alpha_1}(\log N)^\gamma,\quad
		H_2 = N^{1-1/\alpha_2}(\log N)^\gamma.
		\label{eqH'}
	\end{equation}
	By $N\ge\exp(4^{1/\gamma})$, both $H_1$ and $H_2$ are greater than or equal to $4$.
	Noting $\phi_{\alpha_i}(c_{i+2}N) \ll_\epsilon N^{1/\alpha_i-1}$, 
	we have 
	\[
	\phi_{\alpha_i}(c_{i+2}N) + H_i^{-1}
	\ll_\epsilon N^{1/\alpha_i-1} + N^{1/\alpha_i-1}(\log N)^{-\gamma}
	\ll N^{1/\alpha_i-1}
	\]
	for $i=1,2$. This and Lemma~\ref{Koksma} imply that 
	\begin{align*}
		R &\ll_\epsilon \abs{(c_2-c_1)N - \#(\cI\cap\mathbb{Z})}\phi_{\alpha_1}(c_3N)\phi_{\alpha_2}(c_4N)\\
		&\quad+ \sum_{\substack{1\le\abs{h_1}\le H_1 \\ 1\le\abs{h_2}\le H_2}}
		\Bigl( \prod_{i=1,2} \min\{ N^{1/\alpha_i-1}, \abs{h_i}^{-1} \} \Bigr)
		\abs{\sum_{n\in\cI\cap\mathbb{Z}} e\bigl( h_1f_1(n)+h_2f_2(n) \bigr)}\\
		&\quad+ N^{1/\alpha_2-1}\sum_{1\le h\le H_1} \min\{ N^{1/\alpha_1-1}, h^{-1} \}\abs{\sum_{n\in\cI\cap\mathbb{Z}} e\bigl( hf_1(n) \bigr)}\\
		&\quad+ N^{1/\alpha_1-1}\sum_{1\le h\le H_2} \min\{ N^{1/\alpha_2-1}, h^{-1} \}\abs{\sum_{n\in\cI\cap\mathbb{Z}} e\bigl( hf_2(n) \bigr)}\\
		&\quad+ \#(\cI\cap\mathbb{Z})\Bigl( \frac{N^{1/\alpha_1-1}}{H_2} + \frac{N^{1/\alpha_2-1}}{H_1} + \frac{1}{H_1H_2} \Bigr),
	\end{align*}
	where $f_1(x)=x^{1/\alpha_1}$, $f_2(x)=(N-x)^{1/\alpha_2}$, and $\cI=(c_1N, c_2N]$.
	Denote by $R_{1,1}$ (resp.\ $R_3$, $R_{2,1}$, $R_{2,2}$, $R_{1,2}$) 
	the first (resp.\ second, third, fourth, fifth) line of the right-hand side: 
	\[
	R \ll R_{1,1} + R_3 + R_{2,1} + R_{2,2} + R_{1,2}.
	\]
	Also, partition the sum $R_3$ into two sums: 
	\begin{align*}
		R_3 = \sum_{\substack{1\le\abs{h_1}\le H_1 \\ 1\le\abs{h_2}\le H_2}}
		= \sum_{\substack{1\le\abs{h_1}\le H_1 \\ 1\le\abs{h_2}\le H_2 \\ h_1h_2>0}}
		+ \sum_{\substack{1\le\abs{h_1}\le H_1 \\ 1\le\abs{h_2}\le H_2 \\ h_1h_2<0}}
		= R_{3,1}+R_{3,2},\quad\text{say}.
	\end{align*}
	\par
	\setcounter{count}{0}
	\textbf{Step~\num.}
	The value $R_1 \coloneqq R_{1,1}+R_{1,2}$ satisfies \eqref{eqR1'}, 
	since the inequalities 
	\begin{gather*}
		R_{1,1} \ll_\epsilon N^{1/\alpha_1+1/\alpha_2-2}
		\ll_\gamma N^{1/\alpha_1+1/\alpha_2-1}(\log N)^{-\gamma} \quad\text{and}\\
		R_{1,2} \ll N^{1/\alpha_1+1/\alpha_2-1}(\log N)^{-\gamma}
	\end{gather*}
	hold by \eqref{eqH'}.
	\par
	\textbf{Step~\num.}
	In the same way as the proof of Lemma~\ref{lem02}, 
	it follows that 
	\begin{align*}
		R_{2,1} &\ll_{\alpha_1,\alpha_2,\gamma} N^{1/\alpha_2-1/2}(\log N)^{\gamma/2} + N^{1/\alpha_2-1}\log N\\
		&\ll N^{1/\alpha_{\min}-1/2}(\log N)^{\gamma/2},\\
		R_{2,2} &\ll_{\alpha_1,\alpha_2,\gamma} N^{1/\alpha_1-1/2}(\log N)^{\gamma/2} + N^{1/\alpha_1-1}\log N\\
		&\ll N^{1/\alpha_{\min}-1/2}(\log N)^{\gamma/2}.
	\end{align*}
	Therefore, the value $R_2 \coloneqq R_{2,1}+R_{2,2}$ satisfies \eqref{eqR2'}.
	\par
	\textbf{Step~\num.}
	In the same way as the proof of Lemma~\ref{lem02}, 
	it follows that 
	\begin{equation*}
		R_{3,1} \ll_{\alpha_1,\alpha_2,\gamma} N^{1/2}(\log N)^{\gamma/2+1}.
	\end{equation*}
	\par
	\textbf{Step~\num.}
	Let $h_1h_2<0$. Since $H_1$ and $H_2$ are large, we cannot use Lemma~\ref{1stderiv}.
	Instead of Lemma~\ref{1stderiv}, we use Lemma~\ref{3rdderiv} here.
	When $x\in(0,N/2]\cap(2^j,2^{j+1}]$, we have 
	\[
	\abs{f'''(x)} \asymp_{\alpha_1,\alpha_2} \abs{h_1}(2^j)^{1/\alpha_1-3}+\abs{h_2}N^{1/\alpha_2-3}.
	\]
	This and Lemma~\ref{3rdderiv} imply that 
	\begin{align*}
		&\quad \sum_{c_1N<n\le\min\{ c_2N, N/2 \}} e(f(n))\\
		&\ll \sum_{j=\lfloor{\log_2(c_1N)}\rfloor}^{\lfloor{\log_2(c_2N)}\rfloor} \abs{\sum_{\max\{2^j,c_1N\}<n\le\min\{2^{j+1}, c_2N, N/2\}} e(f(n))}\\
		&\ll_{\alpha_1,\alpha_2} \sum_{j=\lfloor{\log_2(c_1N)}\rfloor}^{\lfloor{\log_2(c_2N)}\rfloor}
		\Bigl( 2^j\bigl( \abs{h_1}^{1/6}(2^j)^{1/6\alpha_1-1/2}+\abs{h_2}^{1/6}N^{1/6\alpha_2-1/2} \bigr)\\
		&\qquad\qquad\qquad\qquad+ \min\{ \abs{h_1}^{-1/3}(2^j)^{1-1/3\alpha_1}, \abs{h_2}^{-1/3}N^{1-1/3\alpha_2} \} \Bigr)\\
		&\ll \abs{h_1}^{1/6}N^{1/6\alpha_1+1/2} + \abs{h_2}^{1/6}N^{1/6\alpha_2+1/2}
		+ \abs{h_1}^{-1/3}N^{1-1/3\alpha_1}.
	\end{align*}
	Similarly, 
	\begin{align*}
		&\quad \sum_{\max\{ c_1N, N/2 \}<n\le c_2N} e(f(n))\\
		&\ll_{\alpha_1,\alpha_2} \abs{h_1}^{1/6}N^{1/6\alpha_1+1/2} + \abs{h_2}^{1/6}N^{1/6\alpha_2+1/2}
		+ \abs{h_2}^{-1/3}N^{1-1/3\alpha_2}.
	\end{align*}
	Thus, 
	\begin{align*}
		R_{3,2} &\ll_{\alpha_1,\alpha_2} \sum_{\substack{1\le h_1\le H_1 \\ 1\le h_2\le H_2}}
		\Bigl( \prod_{i=1,2} \min\{ N^{1/\alpha_i-1}, \abs{h_i}^{-1} \} \Bigr)\\
		&\qquad\qquad\qquad ( \abs{h_1}^{1/6}N^{1/6\alpha_1+1/2} + \abs{h_2}^{1/6}N^{1/6\alpha_2+1/2}\\
		&\qquad\qquad\qquad + \abs{h_1}^{-1/3}N^{1-1/3\alpha_1} + \abs{h_2}^{-1/3}N^{1-1/3\alpha_2} ).
	\end{align*}
	By the inequalities 
	\begin{align*}
		\prod_{i=1,2} \min\{ N^{1/\alpha_i-1}, \abs{h_i}^{-1} \}	
		\le \abs{h_1}^{-1}\abs{h_2}^{-1},\ N^{1/\alpha_1-1}\abs{h_2}^{-1},\ \abs{h_1}^{-1}N^{1/\alpha_2-1},
	\end{align*}
	it turns out that 
	\begin{align*}
		R_{3,2} &\ll_{\alpha_1,\alpha_2} \sum_{\substack{1\le h_1\le H_1 \\ 1\le h_2\le H_2}}
		\bigl( h_1^{-1}h_2^{-1}( \abs{h_1}^{1/6}N^{1/6\alpha_1+1/2} + \abs{h_2}^{1/6}N^{1/6\alpha_2+1/2} )\\
		&\qquad\qquad
		+ N^{1/\alpha_1-1}h_2^{-1}\cdot\abs{h_1}^{-1/3}N^{1-1/3\alpha_1} + h_1^{-1}N^{1/\alpha_2-1}\cdot\abs{h_2}^{-1/3}N^{1-1/3\alpha_2} \bigr)\\
		&\ll H_1^{1/6}(\log H_2)N^{1/6\alpha_1+1/2} + (\log H_1)H_2^{1/6}N^{1/6\alpha_2+1/2}\\
		&\quad + H_1^{2/3}(\log H_2)N^{2/3\alpha_1} + (\log H_1)H_2^{2/3}N^{2/3\alpha_2}\\
		&\ll_{\alpha_1,\alpha_2,\gamma} N^{2/3}(\log N)^{2\gamma/3+1},
	\end{align*}
	where we have used \eqref{eqH'} to obtain the last inequality.
	\par
	\textbf{Step~\num.}
	By steps~3--4, the value $R_3=R_{3,1}+R_{3,2}$ satisfies \eqref{eqR3'}.
\end{proof}

\begin{lemma}\label{lem03}
	Let $\alpha_1,\alpha_2\in(1,2)$, $1/\alpha_1+1/\alpha_2>3/2$, and $\epsilon\in(0,1/2)$.
	Then 
	\begin{equation}
	\begin{split}
		\limsup_{N\to\infty} \frac{1}{N^{1/\alpha_1+1/\alpha_2-1}}\#\biggl\{ 1\le n\le\epsilon N : 
		\begin{array}{c}
			\{-n^{1/\alpha_1}\}<\phi_{\alpha_1}(n),\\
			\{-(N-n)^{1/\alpha_2}\}<\phi_{\alpha_2}(N-n)
		\end{array}
		\biggr\}\\
		\ll \epsilon^{1/\alpha_1}(1-\epsilon)^{1/\alpha_2-1} + \epsilon^{1/10}
	\end{split}\label{eq07}
	\end{equation}
	and 
	\begin{equation}
	\begin{split}
		\limsup_{N\to\infty} \frac{1}{N^{1/\alpha_1+1/\alpha_2-1}}\#\biggl\{ N-\epsilon N\le n<N : 
		\begin{array}{c}
			\{-n^{1/\alpha_1}\}<\phi_{\alpha_1}(n),\\
			\{-(N-n)^{1/\alpha_2}\}<\phi_{\alpha_2}(N-n)
		\end{array}
		\biggr\}\\
		\ll \epsilon^{1/\alpha_2}(1-\epsilon)^{1/\alpha_1-1} + \epsilon^{1/10},
	\end{split}\label{eq08}
	\end{equation}
	where the implicit constants are absolute.
\end{lemma}
\begin{proof}
	Eq.~\eqref{eq08} follows from \eqref{eq07} by the symmetry of $\alpha_1$ and $\alpha_2$.
	Hence, we only show \eqref{eq07}.
	Set $(x_1,x_2)=(1/\alpha_1,1/\alpha_2)$.
	By Lemmas~\ref{lem01'} and \ref{lem01''}, 
	there exist real numbers $\beta_1$, $\beta_2$, and $2-1/\alpha_1-1/\alpha_2<\hat{\gamma}_0<1$ 
	with \eqref{eq:beta1}, \eqref{eq:beta2}, and inequalities~1--10 in Lemma~\ref{lem01''}.
	Setting $\gamma_0=1-\hat{\gamma}_0$, 
	we have $0<\gamma_0<1/\alpha_1+1/\alpha_2-1$.
	Since $\phi_{\alpha_i}$ is a decreasing function for $i=1,2$, 
	it follows that 
	\begin{equation}
	\begin{split}
		&\quad \#\biggl\{ 1\le n\le\epsilon N : 
		\begin{array}{c}
			\{-n^{1/\alpha_1}\}<\phi_{\alpha_1}(n),\\
			\{-(N-n)^{1/\alpha_2}\}<\phi_{\alpha_2}(N-n)
		\end{array}
		\biggr\}\\
		&\le N^{\gamma_0} + \sum_{j=\lfloor{\gamma_0\log_2 N}\rfloor}^{\lfloor{\log_2(\epsilon N)}\rfloor}
		\#\biggl\{ 2^j<n\le2^{j+1} : 
		\begin{array}{c}
			\{-n^{1/\alpha_1}\}<\phi_{\alpha_1}(2^j),\\
			\{-(N-n)^{1/\alpha_2}\}<\phi_{\alpha_2}(N-\epsilon N)
		\end{array}
		\biggr\}.
	\end{split}\label{eq10}
	\end{equation}
	\par
	Let $N\ge\epsilon^{-1}$ be an integer, and $j$ be an integer in the range of the above sum.
	Set $r_j=2^j/N$.
	We use Lemma~\ref{lem02} with $c_1=c_3=r_j$, $c_2=2r_j$, and $c_4=1-\epsilon$.
	Since $\hat{c}_1,\hat{c}_2,c_4\asymp 1$ and $c_2\asymp r_j=c_2-c_1=c_1=c_3$, 
	for every integer 
	\begin{equation}
		N \ge \max\{ r_j^{-1}(16\alpha_1^{-1}(1-\epsilon)^{-\beta_1})^{\alpha_1/(\alpha_1-1)},\ 
		(1-2r_j)^{-1}(16\alpha_2^{-1}r_j^{-\beta_2})^{\alpha_2/(\alpha_2-1)} \},
		\label{eqN'}
	\end{equation}
	the value 
	\begin{equation}
	\begin{split}
		R_j &= \#\biggl\{ 2^j<n\le2^{j+1} : 
		\begin{array}{c}
			\{-n^{1/\alpha_1}\} < \phi_{\alpha_1}(2^j),\\
			\{-(N-n)^{1/\alpha_2}\} < \phi_{\alpha_2}(N-\epsilon N)
		\end{array}
		\biggr\}\\
		&\quad- 2^j\phi_{\alpha_1}(2^j)\phi_{\alpha_2}(N-\epsilon N)
	\end{split}\label{eq11}
	\end{equation}
	satisfies the following inequalities: 
	\[
	R_j \ll R_{j,1} + R_{j,2} + R_{j,3},
	\]
	\begin{equation*}
		R_{j,1} \ll r_j(r_j^{1/\alpha_1-1-\beta_2}+r_j^{1/\alpha_1-1})N^{1/\alpha_1+1/\alpha_2-1},
	\end{equation*}
	\begin{equation}
	\begin{split}
		R_{j,2} &\ll_{\alpha_1,\alpha_2} C_2r_j^{1/2\alpha_1}r_j^{(1/2)(1-1/\alpha_1)}N^{1/\alpha_2-1/2}
		+ C_1r_j^{\beta_2/2}N^{1/\alpha_1-1/2}\\
		&\qquad+ C_1C_2( r_j^{1-1/\alpha_1}N^{1/\alpha_2-1} + N^{1/\alpha_1-1} )\log N,\quad\text{and}
	\end{split}\label{eqR2''}
	\end{equation}
	\begin{equation}
	\begin{split}
		R_{j,3} &\ll_{\alpha_1,\alpha_2} ( 1+C_1r_j^{1/2\alpha_1}r_j^{(1/2)(1-1/\alpha_1)}+C_2r_j^{\beta_2/2} )N^{1/2}\log N\\
		&\qquad+ \min\{ C_1r_j^{1-1/\alpha_1}, C_2 \}(\log N)^2,
	\end{split}\label{eqR3''}
	\end{equation}
	where 
	\begin{equation}
	\begin{split}
		C_1 &= \max\{ r_j^{1/\alpha_1-1},\ (1-\epsilon)^{-\beta_1}r_j^{1/\alpha_1-1} \} \ll r_j^{1/\alpha_1-1} \quad\text{and}\\
		C_2 &= \max\{ (1-\epsilon)^{1/\alpha_2-1},\ r_j^{-\beta_2}(1-2r_j)^{1/\alpha_2-1} \} \ll_\epsilon r_j^{-\beta_2}.
	\end{split}\label{eq09} 
	\end{equation}
	(Note that $R_{j,0}=0$ because $2^j$ and $2^{j+1}$ are integers.)
	If $N$ is sufficiently large, then \eqref{eqN'} holds.
	Indeed, 
	\[
	r_j^{-1} \ll ( N^{\gamma_0-1} )^{-1} =o(N) \quad (N\to\infty);
	\]
	moreover, by $1-2r_j\ge1-2\epsilon$, $\beta_2>0$ of \eqref{eq:beta2}, and inequality~6 in Lemma~\ref{lem01''}, 
	we have 
	\begin{align*}
		(1-2r_j)^{-1}r_j^{-\beta_2\alpha_2/(\alpha_2-1)}
		&\ll_{\alpha_2,\beta_2,\epsilon} (N^{\gamma_0-1})^{-\beta_2\alpha_2/(\alpha_2-1)}\\
		&= N^{\hat{\gamma}_0\beta_2\alpha_2/(\alpha_2-1)}
		= o(N) \quad (N\to\infty).
	\end{align*}
	\par
	Denote by $R_{j,2,1}$ (resp.\ $R_{j,2,2}$) the first (resp.\ second) line of the right-hand side of \eqref{eqR2''}, and 
	by $R_{j,3,1}$ (resp.\ $R_{j,3,2}$) the first (resp.\ second) line of the right-hand side of \eqref{eqR3''}: 
	\[
	R_{j,2} \ll_{\alpha_1,\alpha_2} R_{j,2,1}+R_{j,2,2} \quad\text{and}\quad
	R_{j,3} \ll_{\alpha_1,\alpha_2} R_{j,3,1}+R_{j,3,2}.
	\]
	Also, write the sum of $R_{j,1}$ (resp.\ $R_{j,2,1}$, $R_{j,2,2}$, $R_{j,3,1}$, $R_{j,3,2}$) 
	over $\lfloor{\gamma_0\log_2 N}\rfloor\le j\le\lfloor{\log_2(\epsilon N)}\rfloor$ 
	as $S_1$ (resp.\ $S_{2,1}$, $S_{2,2}$, $S_{3,1}$, $S_{3,2}$): 
	\[
	\sum_{j=\lfloor{\gamma_0\log_2 N}\rfloor}^{\lfloor{\log_2(\epsilon N)}\rfloor} R_j
	\ll S_1+(S_{2,1}+S_{2,2})+(S_{3,1}+S_{3,2}).
	\]
	First, we estimate $S_1$.
	By $\beta_2 < (3-1/\alpha_1)^{-1}$ of \eqref{eq:beta12} and $\alpha_1\in(1,2)$, 
	the inequality $1/\alpha_1-\beta_2 > 1/\alpha_1-(3-1/\alpha_1)^{-1} > 1/10$ holds.
	Thus, 
	\[
	\sum_{j=\lfloor{\gamma_0\log_2 N}\rfloor}^{\lfloor{\log_2(\epsilon N)}\rfloor} r_j(r_j^{1/\alpha_1-1-\beta_2}+r_j^{1/\alpha_1-1})
	\ll \epsilon^{1/\alpha_1-\beta_2}+\epsilon^{1/\alpha_1}
	\ll \epsilon^{1/10},
	\]
	which yields that $S_1 \ll \epsilon^{1/10}N^{1/\alpha_1+1/\alpha_2-1}$.
	\par
	We estimate $S_{2,1}$. By \eqref{eq09}, 
	\begin{equation*}
	\begin{split}
		S_{2,1} &\ll_\epsilon \sum_{j=\lfloor{\gamma_0\log_2 N}\rfloor}^{\lfloor{\log_2(\epsilon N)}\rfloor}
		r_j^{1/2-\beta_2}N^{1/\alpha_2-1/2}
		+ \sum_{j=\lfloor{\gamma_0\log_2 N}\rfloor}^{\lfloor{\log_2(\epsilon N)}\rfloor}
		r_j^{\beta_2/2+1/\alpha_1-1}N^{1/\alpha_1-1/2}\\
		&= S_{2,1,1}+S_{2,1,2}, \quad\text{say}.
	\end{split}
	\end{equation*}
	If $\beta_2\le1/2$, then $S_{2,1,1} \ll N^{1/\alpha_2-1/2}\log N$; 
	if $\beta_2>1/2$, then 
	\[
	S_{2,1,1} \ll_{\beta_2} (N^{\gamma_0-1})^{1/2-\beta_2}N^{1/\alpha_2-1/2}
	= N^{\hat{\gamma}_0(\beta_2-1/2)}N^{1/\alpha_2-1/2}.
	\]
	By $\alpha_1\in(1,2)$ and inequality~9 in Lemma~\ref{lem01''}, 
	it follows that $S_{2,1,1}=o(N^{1/\alpha_1+1/\alpha_2-1})$ in both cases of $\beta_2\le1/2$ and $\beta_2>1/2$.
	Also, 
	\begin{equation*}
		S_{2,1,2} \ll_{\alpha_1,\beta_2} (N^{\gamma_0-1})^{\beta_2/2+1/\alpha_1-1}N^{1/\alpha_1-1/2}
		= N^{-\hat{\gamma}_0(\beta_2/2+1/\alpha_1-1)}N^{1/\alpha_1-1/2},
	\end{equation*}
	since $\beta_2/2+1/\alpha_1-1<0$ by \eqref{eq:beta2}.
	By inequality~7 in Lemma~\ref{lem01''}, we have $S_{2,1,2}=o(N^{1/\alpha_1+1/\alpha_2-1})$.
	Therefore, $S_{2,1}=o(N^{1/\alpha_1+1/\alpha_2-1})$ as $N\to\infty$.
	\par
	We estimate $S_{2,2}$. By \eqref{eq09}, 
	\begin{align*}
		\frac{S_{2,2}}{\log N} &= \sum_{j=\lfloor{\gamma_0\log_2 N}\rfloor}^{\lfloor{\log_2(\epsilon N)}\rfloor}
		C_1C_2( r_j^{1-1/\alpha_1}N^{1/\alpha_2-1} + N^{1/\alpha_1-1} )\\
		&\ll_\epsilon \sum_{j=\lfloor{\gamma_0\log_2 N}\rfloor}^{\lfloor{\log_2(\epsilon N)}\rfloor}
		( r_j^{-\beta_2}N^{1/\alpha_2-1} + r_j^{1/\alpha_1-1-\beta_2}N^{1/\alpha_1-1} ).
	\end{align*}
	Since $\beta_2>0$ and $1/\alpha_1-1<0$ by \eqref{eq:beta2} and $\alpha_1\in(1,2)$, 
	we have 
	\begin{align*}
		\frac{S_{2,2}}{\log N} &\ll_\epsilon \sum_{j=\lfloor{\gamma_0\log_2 N}\rfloor}^{\lfloor{\log_2(\epsilon N)}\rfloor}
		( r_j^{-\beta_2}N^{1/\alpha_2-1} + r_j^{1/\alpha_1-1-\beta_2}N^{1/\alpha_1-1} )\\
		&\ll_{\alpha_1,\beta_2} N^{\hat{\gamma}_0\beta_2}N^{1/\alpha_2-1} + N^{\hat{\gamma}_0(1-1/\alpha_1+\beta_2)}N^{1/\alpha_1-1}.
	\end{align*}
	By inequalities~6 and 10 in Lemma~\ref{lem01''}, 
	it follows that $S_{2,2}=o(N^{1/\alpha_1+1/\alpha_2-1})$ as $N\to\infty$.
	\par
	We estimate $S_{3,1}$. By \eqref{eq09}, 
	\begin{align*}
		\frac{S_{3,1}}{N^{1/2}\log N} &= \sum_{j=\lfloor{\gamma_0\log_2 N}\rfloor}^{\lfloor{\log_2(\epsilon N)}\rfloor}
		( 1+C_1r_j^{1/2\alpha_1}r_j^{(1/2)(1-1/\alpha_1)}+C_2r_j^{\beta_2/2} )\\
		&\ll_\epsilon \sum_{j=\lfloor{\gamma_0\log_2 N}\rfloor}^{\lfloor{\log_2(\epsilon N)}\rfloor}
		( 1+r_j^{1/\alpha_1-1/2}+r_j^{-\beta_2/2} ).
	\end{align*}
	Since $1/\alpha_1-1/2>0$ and $\beta_2>0$ by $\alpha_1\in(1,2)$ and \eqref{eq:beta2}, 
	\begin{align*}
		\frac{S_{3,1}}{N^{1/2}\log N} &\ll_\epsilon \sum_{j=\lfloor{\gamma_0\log_2 N}\rfloor}^{\lfloor{\log_2(\epsilon N)}\rfloor}
		( 1+r_j^{1/\alpha_1-1/2}+r_j^{-\beta_2/2} )\\
		&\ll_{\alpha_1,\beta_2} \log N + \epsilon^{1/\alpha_1-1/2} + N^{\hat{\gamma}_0\beta_2/2}
		\ll_{\beta_2,\gamma_0} N^{\hat{\gamma}_0\beta_2/2}.
	\end{align*}
	By inequality~8 in Lemma~\ref{lem01''}, it follows that $S_{3,1}=o(N^{1/\alpha_1+1/\alpha_2-1})$ as $N\to\infty$.
	\par
	We estimate $S_{3,2}$. By \eqref{eq09}, 
	\[
	\frac{S_{3,2}}{(\log N)^2} = \sum_{j=\lfloor{\gamma_0\log_2 N}\rfloor}^{\lfloor{\log_2(\epsilon N)}\rfloor}
	\min\{ C_1r_j^{1-1/\alpha_1}, C_2 \}
	\ll_\epsilon \sum_{j=\lfloor{\gamma_0\log_2 N}\rfloor}^{\lfloor{\log_2(\epsilon N)}\rfloor}
	\min\{ 1, r_j^{-\beta_2} \} \ll \log N.
	\]
	Thus, $S_{3,2}=o(N^{1/\alpha_1+1/\alpha_2-1})$ as $N\to\infty$.
	\par
	By the above results, 
	\begin{equation}
	\begin{split}
		&\quad \limsup_{N\to\infty} \frac{1}{N^{1/\alpha_1+1/\alpha_2-1}}
		\sum_{j=\lfloor{\gamma_0\log_2 N}\rfloor}^{\lfloor{\log_2(\epsilon N)}\rfloor} R_j\\
		&\ll \limsup_{N\to\infty} \frac{S_1+(S_{2,1}+S_{2,2})+(S_{3,1}+S_{3,2})}{N^{1/\alpha_1+1/\alpha_2-1}}
		\ll \epsilon^{1/10}.
	\end{split}\label{eq12}
	\end{equation}
	Since $\phi_{\alpha_i}(x) \ll x^{1/\alpha_i-1}$ for $i=1,2$, we have 
	\begin{equation}
	\begin{split}
		&\quad \sum_{j=\lfloor{\gamma_0\log_2 N}\rfloor}^{\lfloor{\log_2(\epsilon N)}\rfloor}
		2^j\phi_{\alpha_1}(2^j)\phi_{\alpha_2}(N-\epsilon N)\\
		&\ll \sum_{j=\lfloor{\gamma_0\log_2 N}\rfloor}^{\lfloor{\log_2(\epsilon N)}\rfloor}
		2^j\cdot(2^j)^{1/\alpha_1-1}(N-\epsilon N)^{1/\alpha_2-1}\\
		&\ll (\epsilon N)^{1/\alpha_1}(N-\epsilon N)^{1/\alpha_2-1}
		= \epsilon^{1/\alpha_1}(1-\epsilon)^{1/\alpha_2-1}N^{1/\alpha_1+1/\alpha_2-1}.
	\end{split}\label{eq13}
	\end{equation}
	Eq.~\eqref{eq07} follows from \eqref{eq10}, \eqref{eq11}, \eqref{eq12}, and \eqref{eq13}.
\end{proof}

\begin{lemma}\label{lem03'}
	Let $\alpha_1,\alpha_2\in(1,2)$, $1/\alpha_1+1/\alpha_2>5/3$, and $\epsilon\in(0,1/2)$.
	Then 
	\begin{align*}
		\lim_{N\to\infty} \frac{1}{N^{1/\alpha_1+1/\alpha_2-1}}\#\biggl\{ \epsilon N<n\le(1-\epsilon)N : 
		\begin{array}{c}
			\{-n^{1/\alpha_1}\}<\phi_{\alpha_1}(n),\\
			\{-(N-n)^{1/\alpha_2}\}<\phi_{\alpha_2}(N-n)
		\end{array}
		\biggr\}\\
		= \alpha_1^{-1}\alpha_2^{-1}\int_\epsilon^{1-\epsilon} x^{1/\alpha_1-1}(1-x)^{1/\alpha_2-1}\,dx.
	\end{align*}
\end{lemma}
\begin{proof}
	Let $m\ge2$ and $N\ge\max\{ \epsilon^{-1}, (m-1)(1-2\epsilon)^{-1} \}$ be integers.
	Partition the interval $(0,1]$ into the intervals $(s_j,s_{j+1}]$, $j=0,\ldots,m$, with the following conditions: 
	(i) $s_0=0$ and $s_{m+1}=1$; 
	(ii) $s_i=\epsilon+(1-2\epsilon)(m-1)^{-1}(j-1)$ for $j=1,\ldots,m$.
	Then $1/N\le s_{j+1}-s_j<1/(m-1)$ for every $j=1,2,\ldots,m-1$.
	Since $\phi_{\alpha_i}$ is a decreasing function for $i=1,2$, 
	it follows that 
	\begin{equation}
	\begin{split}
		&\quad \#\biggl\{ \epsilon N<n\le(1-\epsilon)N : 
		\begin{array}{c}
			\{-n^{1/\alpha_1}\}<\phi_{\alpha_1}(n),\\
			\{-(N-n)^{1/\alpha_2}\}<\phi_{\alpha_2}(N-n)
		\end{array}
		\biggr\}\\
		&= \sum_{j=1}^{m-1} \#\biggl\{ s_jN<n\le s_{j+1}N : 
		\begin{array}{c}
			\{-n^{1/\alpha_1}\}<\phi_{\alpha_1}(n),\\
			\{-(N-n)^{1/\alpha_2}\}<\phi_{\alpha_2}(N-n)
		\end{array}
		\biggr\}\\
		&\le \sum_{j=1}^{m-1} \#\biggl\{ s_jN<n\le s_{j+1}N : 
		\begin{array}{c}
			\{-n^{1/\alpha_1}\}<\phi_{\alpha_1}(s_jN),\\
			\{-(N-n)^{1/\alpha_2}\}<\phi_{\alpha_2}(N-s_{j+1}N)
		\end{array}
		\biggr\}.
	\end{split}\label{eq14}
	\end{equation}
	\par
	Let $1\le j\le m-1$ be an integer.
	We use Lemma~\ref{lem02'} with $c_1=c_3=s_j$, $c_2=s_{j+1}$, $c_4=1-s_{j+1}$, and $\gamma=1$.
	(Note that $c_1,c_2,c_3,c_4\in[\epsilon,1-\epsilon]$, $c_1,c_2\in[1/N,1-1/N]$, and $c_2-c_1\ge1/N$.)
	For every integer $N\ge e^4$, the value 
	\begin{align*}
		R_j &= \#\biggl\{ s_jN<n\le s_{j+1}N : 
		\begin{array}{c}
			\{-n^{1/\alpha_1}\}<\phi_{\alpha_1}(s_jN),\\
			\{-(N-n)^{1/\alpha_2}\}<\phi_{\alpha_2}(N-s_{j+1}N)
		\end{array}
		\biggr\}\\
		&\quad- (s_{j+1}-s_j)N\phi_{\alpha_1}(s_jN)\phi_{\alpha_2}(N-s_{j+1}N)
	\end{align*}
	satisfies the inequality 
	\begin{equation*}
	\begin{split}
		R_j &\ll_{\alpha_1,\alpha_2,\epsilon} N^{1/\alpha_1+1/\alpha_2-1}(\log N)^{-1}\\
		&\qquad\qquad+ N^{1/\alpha_{\min}-1/2}(\log N)^{1/2}\\
		&\qquad\qquad+ N^{2/3}(\log N)^{5/3}
	\end{split}
	\end{equation*}
	This inequality reduces to 
	\[
	R_j \ll_{\alpha_1,\alpha_2,\epsilon} N^{1/\alpha_1+1/\alpha_2-1}(\log N)^{-1}
	\]
	by $\alpha_1,\alpha_2\in(1,2)$ and $1/\alpha_1+1/\alpha_2>5/3$.
	Since $0<\phi_{\alpha_i}(x)<\alpha_i^{-1}x^{1/\alpha_i-1}$ for $i=1,2$, 
	it follows that 
	\begin{align*}
		&\quad \sum_{j=1}^{m-1} \#\biggl\{ s_jN<n\le s_{j+1}N : 
		\begin{array}{c}
			\{-n^{1/\alpha_1}\}<\phi_{\alpha_1}(s_jN),\\
			\{-(N-n)^{1/\alpha_2}\}<\phi_{\alpha_2}(N-s_{j+1}N)
		\end{array}
		\biggr\}\\
		&= \sum_{j=1}^{m-1} (s_{j+1}-s_j)N\phi_{\alpha_1}(s_jN)\phi_{\alpha_2}(N-s_{j+1}N)
		+ \sum_{j=1}^{m-1} R_j\\
		&\le \alpha_1^{-1}\alpha_2^{-1}N^{1/\alpha_1+1/\alpha_2-1}\sum_{j=1}^{m-1} (s_{j+1}-s_j)s_j^{1/\alpha_1-1}(1-s_{j+1})^{1/\alpha_2-1}\\
		&\quad+ O_{\alpha_1,\alpha_2,\epsilon,m}\bigl( N^{1/\alpha_1+1/\alpha_2-1}(\log N)^{-1} \bigr).		
	\end{align*}
	This and \eqref{eq14} yield that 
	\begin{align*}
		&\quad \limsup_{N\to\infty} \frac{1}{N^{1/\alpha_1+1/\alpha_2-1}}\#\biggl\{ \epsilon N<n\le (1-\epsilon)N : 
		\begin{array}{c}
			\{-n^{1/\alpha_1}\}<\phi_{\alpha_1}(n),\\
			\{-(N-n)^{1/\alpha_2}\}<\phi_{\alpha_2}(N-n)
		\end{array}
		\biggr\}\\
		&\le \alpha_1^{-1}\alpha_2^{-1}\sum_{j=1}^{m-1} (s_{j+1}-s_j)s_j^{1/\alpha_1-1}(1-s_{j+1})^{1/\alpha_2-1}.
	\end{align*}
	Similarly, 
	\begin{align*}
		&\quad \liminf_{N\to\infty} \frac{1}{N^{1/\alpha_1+1/\alpha_2-1}}\#\biggl\{ \epsilon N<n\le (1-\epsilon)N : 
		\begin{array}{c}
			\{-n^{1/\alpha_1}\}<\phi_{\alpha_1}(n),\\
			\{-(N-n)^{1/\alpha_2}\}<\phi_{\alpha_2}(N-n)
		\end{array}
		\biggr\}\\
		&\ge \alpha_1^{-1}\alpha_2^{-1}\sum_{j=1}^{m-1} (s_{j+1}-s_j)s_{j+1}^{1/\alpha_1-1}(1-s_j)^{1/\alpha_2-1}.
	\end{align*}
	Letting $m\to\infty$, we obtain the desired equality.
\end{proof}

\begin{lemma}\label{lem04}
	Let $\alpha_1,\alpha_2\in(1,2)$ and $1/\alpha_1+1/\alpha_2>5/3$.
	Then 
	\begin{align*}
		\lim_{N\to\infty} \frac{1}{N^{1/\alpha_1+1/\alpha_2-1}}\#\biggl\{ 1\le n<N : 
		\begin{array}{c}
			\{-n^{1/\alpha_1}\}<\phi_{\alpha_1}(n),\\
			\{-(N-n)^{1/\alpha_2}\}<\phi_{\alpha_2}(N-n)
		\end{array}
		\biggr\}\\
		= \alpha_1^{-1}\alpha_2^{-1}B(1/\alpha_1,1/\alpha_2).
	\end{align*}
\end{lemma}
\begin{proof}
	Let $\epsilon\in(0,1/2)$. Then 
	\begin{align}
		&\quad \#\biggl\{ 1\le n<N : 
		\begin{array}{c}
			\{-n^{1/\alpha_1}\}<\phi_{\alpha_1}(n),\\
			\{-(N-n)^{1/\alpha_2}\}<\phi_{\alpha_2}(N-n)
		\end{array}
		\biggr\}\nonumber\\
		\begin{split}
			&= \#\biggl\{ \epsilon N<n\le(1-\epsilon)N : 
			\begin{array}{c}
				\{-n^{1/\alpha_1}\}<\phi_{\alpha_1}(n),\\
				\{-(N-n)^{1/\alpha_2}\}<\phi_{\alpha_2}(N-n)
			\end{array}
			\biggr\}\\
			&\quad+ \#\biggl\{ 1\le n\le\epsilon N : 
			\begin{array}{c}
				\{-n^{1/\alpha_1}\}<\phi_{\alpha_1}(n),\\
				\{-(N-n)^{1/\alpha_2}\}<\phi_{\alpha_2}(N-n)
			\end{array}
			\biggr\}\\
			&\quad+ \#\biggl\{ (1-\epsilon)N<n<N : 
			\begin{array}{c}
				\{-n^{1/\alpha_1}\}<\phi_{\alpha_1}(n),\\
				\{-(N-n)^{1/\alpha_2}\}<\phi_{\alpha_2}(N-n)
			\end{array}
			\biggr\}.
		\end{split}\label{eq15}
	\end{align}
	By Lemma~\ref{lem03'}, 
	\begin{equation}
		\lim_{N\to\infty} \frac{\text{The first term of \eqref{eq15}}}{N^{1/\alpha_1+1/\alpha_2-1}}
		= \alpha_1^{-1}\alpha_2^{-1}\int_\epsilon^{1-\epsilon} x^{1/\alpha_1-1}(1-x)^{1/\alpha_2-1}\,dx.
		\label{eq16}
	\end{equation}
	By Lemma~\ref{lem03}, 
	\begin{align*}
		\limsup_{N\to\infty} \frac{\text{The second term of \eqref{eq15}}}{N^{1/\alpha_1+1/\alpha_2-1}}
		&\ll \epsilon^{1/\alpha_1}(1-\epsilon)^{1/\alpha_2-1} +\epsilon^{1/10} \quad\text{and}\\
		\limsup_{N\to\infty} \frac{\text{The third term of \eqref{eq15}}}{N^{1/\alpha_1+1/\alpha_2-1}}
		&\ll \epsilon^{1/\alpha_2}(1-\epsilon)^{1/\alpha_1-1} + \epsilon^{1/10}.
	\end{align*}
	Thus, for some absolute constant $C>0$, 
	\begin{align*}
		&\quad \limsup_{N\to\infty} \frac{1}{N^{1/\alpha_1+1/\alpha_2-1}}\#\biggl\{ 1\le n<N : 
		\begin{array}{c}
			\{-n^{1/\alpha_1}\}<\phi_{\alpha_1}(n),\\
			\{-(N-n)^{1/\alpha_2}\}<\phi_{\alpha_2}(N-n)
		\end{array}
		\biggr\}\\
		&\le \alpha_1^{-1}\alpha_2^{-1}\int_\epsilon^{1-\epsilon} x^{1/\alpha_1-1}(1-x)^{1/\alpha_2-1}\,dx\\
		&\quad+ C\bigl( \epsilon^{1/\alpha_1}(1-\epsilon)^{1/\alpha_2-1} + \epsilon^{1/\alpha_2}(1-\epsilon)^{1/\alpha_1-1} + \epsilon^{1/10} \bigr).
	\end{align*}
	Also, from 
	\begin{align*}
		&\quad \#\biggl\{ 1\le n<N : 
		\begin{array}{c}
			\{-n^{1/\alpha_1}\}<\phi_{\alpha_1}(n),\\
			\{-(N-n)^{1/\alpha_2}\}<\phi_{\alpha_2}(N-n)
		\end{array}
		\biggr\}\nonumber\\
		&\ge \#\biggl\{ \epsilon N<n\le(1-\epsilon)N : 
		\begin{array}{c}
			\{-n^{1/\alpha_1}\}<\phi_{\alpha_1}(n),\\
			\{-(N-n)^{1/\alpha_2}\}<\phi_{\alpha_2}(N-n)
		\end{array}
		\biggr\}
	\end{align*}
	and \eqref{eq16}, it follows that 
	\begin{align*}
		\liminf_{N\to\infty} \frac{1}{N^{1/\alpha_1+1/\alpha_2-1}}\#\biggl\{ 1\le n<N : 
		\begin{array}{c}
			\{-n^{1/\alpha_1}\}<\phi_{\alpha_1}(n),\\
			\{-(N-n)^{1/\alpha_2}\}<\phi_{\alpha_2}(N-n)
		\end{array}
		\biggr\}\\
		\ge \alpha_1^{-1}\alpha_2^{-1}\int_\epsilon^{1-\epsilon} x^{1/\alpha_1-1}(1-x)^{1/\alpha_2-1}\,dx.
	\end{align*}
	Letting $\epsilon\to+0$, we obtain the desired equality.
\end{proof}

If considering only a non-asymptotic upper bound instead of an asymptotic result, 
we obtain a larger range of $\alpha_1$ and $\alpha_2$ than that of Lemma~\ref{lem03'}.

\begin{lemma}\label{lem03''}
	Let $\alpha_1,\alpha_2\in(1,2)$, $1/\alpha_1+1/\alpha_2>3/2$, and $\epsilon\in(0,1/2)$.
	Then 
	\begin{align*}
		\limsup_{N\to\infty} \frac{1}{N^{1/\alpha_1+1/\alpha_2-1}}\#\biggl\{ \epsilon N<n\le(1-\epsilon)N : 
		\begin{array}{c}
			\{-n^{1/\alpha_1}\}<\phi_{\alpha_1}(n),\\
			\{-(N-n)^{1/\alpha_2}\}<\phi_{\alpha_2}(N-n)
		\end{array}
		\biggr\}\\
		\ll \int_\epsilon^{1-\epsilon} x^{1/\alpha_1-1}(1-x)^{1/\alpha_2-1}\,dx,
	\end{align*}
	where the implicit constant is absolute.
\end{lemma}
\begin{proof}
	The proof below is similar to that of Lemma~\ref{lem03'}.
	Let $m\ge2$ and $N\ge\max\{ \epsilon^{-1}, (m-1)(1-2\epsilon)^{-1} \}$ be integers.
	Partition the interval $(0,1]$ into the intervals $(s_j,s_{j+1}]$, $j=0,\ldots,m$, with the following conditions: 
	(i) $s_0=0$ and $s_{m+1}=1$; 
	(ii) $s_i=\epsilon+(1-2\epsilon)(m-1)^{-1}(j-1)$ for $j=1,\ldots,m$.
	Then $1/N\le s_{j+1}-s_j<1/(m-1)$ for every $j=1,2,\ldots,m-1$.
	Since $\phi_{\alpha_i}$ is a decreasing function for $i=1,2$, 
	\eqref{eq14} follows.
	\par
	Let $1\le j\le m-1$ be an integer.
	We use Lemma~\ref{lem02} with $c_1=c_3=s_j$, $c_2=s_{j+1}$, $c_4=1-s_{j+1}$, and $\beta_1=\beta_2=0$.
	(Note that $c_1,c_2,c_3,c_4\in[\epsilon,1-\epsilon]$, $c_1,c_2\in[1/N,1-1/N]$, and $c_2-c_1\ge1/N$.)
	For every integer 
	\begin{equation*}
		N \ge \max\{ s_j^{-1}(16\alpha_1^{-1})^{\alpha_1/(\alpha_1-1)},\ 
		(1-s_{j+1})^{-1}(16\alpha_2^{-1})^{\alpha_2/(\alpha_2-1)} \},
	\end{equation*}
	the value 
	\begin{equation}
	\begin{split}
		R_j &= \#\biggl\{ s_jN<n\le s_{j+1}N : 
		\begin{array}{c}
			\{-n^{1/\alpha_1}\}<\phi_{\alpha_1}(s_jN),\\
			\{-(N-n)^{1/\alpha_2}\}<\phi_{\alpha_2}(N-s_{j+1}N)
		\end{array}
		\biggr\}\\
		&\quad- (s_{j+1}-s_j)N\phi_{\alpha_1}(s_jN)\phi_{\alpha_2}(N-s_{j+1}N)
	\end{split}\label{eq18}
	\end{equation}
	satisfies the following inequalities:
	\[
	R_j \ll R_{j,0} + R_{j,1} + R_{j,2} + R_{j,3},
	\]
	\begin{equation*}
		R_{j,0} \ll_\epsilon N^{1/\alpha_1+1/\alpha_2-2},
	\end{equation*}
	\begin{equation*}
		R_{j,1} \ll (s_{j+1}-s_j)s_j^{1/\alpha_1-1}(1-s_{j+1})^{1/\alpha_2-1}N^{1/\alpha_1+1/\alpha_2-1},
	\end{equation*}
	\begin{equation*}
	\begin{split}
		R_{j,2} &\ll_{\alpha_1,\alpha_2,\epsilon} N^{1/\alpha_2-1/2} + N^{1/\alpha_1-1/2}
		+ ( N^{1/\alpha_2-1} + N^{1/\alpha_1-1} )\log N,\quad\text{and}
	\end{split}
	\end{equation*}
	\begin{equation*}
	\begin{split}
		R_{j,3} &\ll_{\alpha_1,\alpha_2,\epsilon} N^{1/2}\log N + (\log N)^2.
	\end{split}
	\end{equation*}
	By $\alpha_1,\alpha_2\in(1,2)$ and $1/\alpha_1+1/\alpha_2>3/2$, we have 
	\begin{equation}
		\limsup_{N\to\infty} \frac{1}{N^{1/\alpha_1+1/\alpha_2-1}}\sum_{j=1}^{m-1} R_j
		\ll \sum_{j=1}^{m-1} (s_{j+1}-s_j)s_j^{1/\alpha_1-1}(1-s_{j+1})^{1/\alpha_2-1}.
		\label{eq19}
	\end{equation}
	Since $0<\phi_{\alpha_i}(x)\ll x^{1/\alpha_i-1}$ for $i=1,2$, it follows that 
	\begin{align*}
		&\quad \sum_{j=1}^{m-1} (s_{j+1}-s_j)N\phi_{\alpha_1}(s_jN)\phi_{\alpha_2}(N-s_{j+1}N)\\
		&\ll N^{1/\alpha_1+1/\alpha_2-1}\sum_{j=1}^{m-1} (s_{j+1}-s_j)s_j^{1/\alpha_1-1}(1-s_{j+1})^{1/\alpha_2-1}.
	\end{align*}
	This, \eqref{eq14}, \eqref{eq18}, and \eqref{eq19} yield that 
	\begin{align*}
		\limsup_{N\to\infty} \frac{1}{N^{1/\alpha_1+1/\alpha_2-1}}\#\biggl\{ \epsilon N<n\le (1-\epsilon)N : 
		\begin{array}{c}
			\{-n^{1/\alpha_1}\}<\phi_{\alpha_1}(n),\\
			\{-(N-n)^{1/\alpha_2}\}<\phi_{\alpha_2}(N-n)
		\end{array}
		\biggr\}\\
		\ll \sum_{j=1}^{m-1} (s_{j+1}-s_j)s_j^{1/\alpha_1-1}(1-s_{j+1})^{1/\alpha_2-1}.
	\end{align*}
	Letting $m\to\infty$, we obtain the desired inequality.
\end{proof}

Using Lemmas~\ref{lem04}, \ref{lem03}, and \ref{lem03''}, 
we prove Theorems~\ref{main1} and \ref{main2}.

\begin{proof}[Proof of Theorem~$\ref{main1}$]
	First, it is clear that 
	\begin{align}
		\cR_{\alpha_1,\alpha_2}(N)
		&= \#\{ 1\le n<N : \exists(n_1,n_2)\in\mathbb{N}^2,\ \lfloor{n_1^{\alpha_1}}\rfloor=n,\ \lfloor{n_2^{\alpha_2}}\rfloor=N-n \}\nonumber\\
		&= \#\biggl\{ 1\le n<N : 
		\begin{array}{l}
			\exists(n_1,n_2)\in\mathbb{N}^2,\ n^{1/\alpha_1}\le n_1<(n+1)^{1/\alpha_1}\text{ and}\\
			(N-n)^{1/\alpha_2}\le n_2<(N-n+1)^{1/\alpha_2}
		\end{array}
		\biggr\}. \label{eq17}
	\end{align}
	Also, if real numbers $a<b$ satisfy $0<b-a<1$, then the following two conditions are equivalent to each other: 
	(i) the interval $(a,b]$ contains exact one integer; (ii) $\{b\}<b-a$. 
	Since $0<\phi_{\alpha_i}(x)<1$ for $i=1,2$, \eqref{eq17} is equal to 
	\begin{equation*}
		\#\biggl\{ 1\le n<N : 
		\begin{array}{c}
			\{-n^{1/\alpha_1}\}<\phi_{\alpha_1}(n),\\
			\{-(N-n)^{1/\alpha_2}\}<\phi_{\alpha_2}(N-n)
		\end{array}
		\biggr\}.
	\end{equation*}
	Therefore, Theorem~\ref{main1} follows from Lemma~\ref{lem04}.
\end{proof}

\begin{proof}[Proof of Theorem~$\ref{main2}$]
	Theorem~\ref{main2} follows from Lemmas~\ref{lem03} and \ref{lem03''} 
	in the same way as the proof of Theorem~\ref{main1}.
\end{proof}

\section{Heuristic argument}\label{heuristic}

In this section, we see asymptotic formulas for the following 
$\cN_\alpha^{(1,2)}(N)$, $\cN_\alpha^{(3)}(N)$, and $\cN_\alpha^{\AP}(N)$ in a heuristic way.
For a real number $\alpha>1$ and an integer $N\ge1$, 
define the numbers $\cN_\alpha^{(1,2)}(N)$, $\cN_\alpha^{(3)}(N)$, and $\cN_\alpha^{\AP}(N)$ as 
\begin{align*}
	\cN_\alpha^{(1,2)}(N)
	&= \#\{ (l,m,n)\in\mathbb{N}^3 : l,m\le N,\ \lfloor{l^\alpha}\rfloor+\lfloor{m^\alpha}\rfloor=\lfloor{n^\alpha}\rfloor \},\\
	\cN_\alpha^{(3)}(N)
	&= \#\{ (l,m,n)\in\mathbb{N}^3 : n\le N,\ \lfloor{l^\alpha}\rfloor+\lfloor{m^\alpha}\rfloor=\lfloor{n^\alpha}\rfloor \},\quad\text{and}\\
	\cN_\alpha^{\AP}(N)
	&= \#\{ (l,m,n)\in\mathbb{N}^3 : l<m<n\le N,\ \lfloor{l^\alpha}\rfloor+\lfloor{n^\alpha}\rfloor=2\lfloor{m^\alpha}\rfloor \},
\end{align*}
respectively. When $\alpha>1$ is close to $1$, 
we have already estimated $\cN_\alpha^{(3)}(N)$ and $\cN_\alpha^{\AP}(N)$ in Corollaries~\ref{main1''}, \ref{main2''}, and \ref{main2'''}.

\begin{conjecture}\label{conj1}
	For every $\alpha\in(1,2)\cup(2,3)$, we have 
	\begin{equation}
	\begin{split}
		\lim_{N\to\infty} \frac{\cN_\alpha^{(1,2)}(N)}{N^{3-\alpha}}
		&= \alpha^{-3}\iint_{0<x,y\le1} \bigl( xy(x+y) \bigr)^{1/\alpha-1}\,dxdy\\
		&= \frac{\Gamma(1+1/\alpha)^2}{(3-\alpha)\Gamma(2/\alpha)} + I(\alpha)
	\end{split}\label{eq:conj1}
	\end{equation}
	and 
	\begin{equation}
	\begin{split}
		\lim_{N\to\infty} \frac{\cN_\alpha^{(3)}(N)}{N^{3-\alpha}}
		&= \alpha^{-3}\iint_{\substack{x,y>0 \\ x+y\le1}} \bigl( xy(x+y) \bigr)^{1/\alpha-1}\,dxdy\\
		&= \frac{\Gamma(1+1/\alpha)^2}{(3-\alpha)\Gamma(2/\alpha)},
	\end{split}\label{eq:conj2}
	\end{equation}
	where the function $I$ is defined as 
	\[
	I(\alpha) = \alpha^{-3}\int_1^2 u^{3/\alpha-2}\,du \int_{1-1/u}^{1/u} \bigl( (1-v)v \bigr)^{1/\alpha-1}\,dv
	\]
	for $\alpha\in(1,3)$.
\end{conjecture}

\begin{conjecture}\label{conj2}
	For every $\alpha\in(1,2)\cup(2,3)$, we have 
	\begin{equation}
	\begin{split}
		\lim_{N\to\infty} \frac{\cN_\alpha^{\AP}(N)}{N^{3-\alpha}}
		&= 2^{-1/\alpha-1}\alpha^{-3}\iint_{0<x,y\le1} \bigl( xy(x+y) \bigr)^{1/\alpha-1}\,dxdy\\
		&= 2^{-1/\alpha-1}\Bigl( \frac{\Gamma(1+1/\alpha)^2}{(3-\alpha)\Gamma(2/\alpha)} + I(\alpha) \Bigr),
	\end{split}\label{eq:conj3}
	\end{equation}
	where $I(\alpha)$ is the same as in Conjecture~$\ref{conj1}$.
\end{conjecture}

Neither Conjecture~\ref{conj1} nor \ref{conj2} includes the case $\alpha=2$.
The asymptotic formulas for $\cN_2^{(1,2)}(N)$, $\cN_2^{(3)}(N)$, and $\cN_2^{\AP}(N)$ 
need a logarithmic factor in leading terms; see \cite{BV}, \eqref{eqN3}, and \eqref{eqNAP}.
This is probably because the functions 
$f_\alpha(l,m,n) = \lfloor{l^\alpha}\rfloor+\lfloor{m^\alpha}\rfloor-\lfloor{n^\alpha}\rfloor$ and 
$g_\alpha(l,m,n) = \lfloor{l^\alpha}\rfloor+\lfloor{n^\alpha}\rfloor-2\lfloor{m^\alpha}\rfloor$ are homogeneous if $\alpha=2$.
However, neither $f_\alpha$ nor $g_\alpha$ is homogeneous if $\alpha>1$ is non-integral.
One can observe that the case $\alpha=2$ is an exception by numerical computation; see Figures~\ref{fig1} and \ref{fig2}.

Comparing Figures~\ref{fig1} and \ref{fig2}, 
we find that some points in Figure~\ref{fig2} are not so near the horizontal line $\mathbb{R}\times\{1\}$.
Hence, we might need to modify the right-hand side of \eqref{eq:conj3}.
However, the non-asymptotic formula $\cN_\alpha^{\AP}(N) \asymp_\alpha N^{3-\alpha}$ is probably true for every $\alpha\in(1,2)$ 
as we have proved it for every $\alpha\in(1,4/3)$ (Corollary~\ref{main2'''}).

\begin{figure}
	\centering
	\includegraphics[scale=0.55]{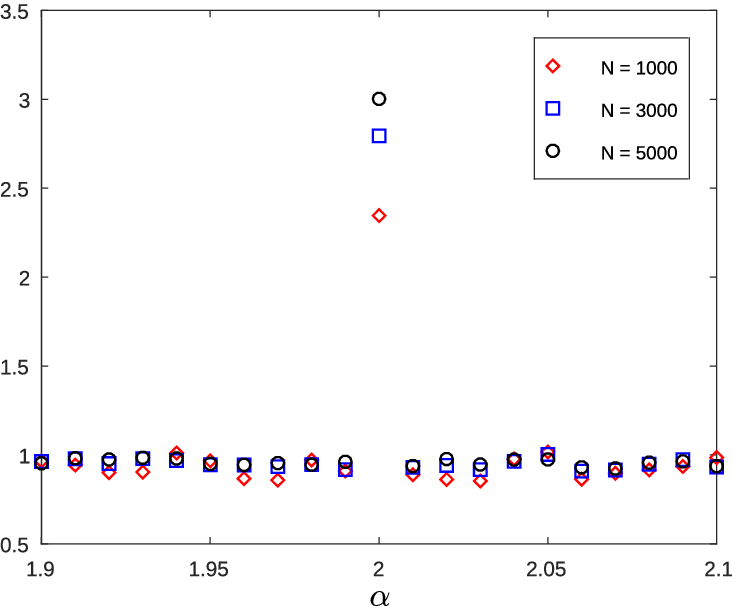}
	\hspace{2em}
	\includegraphics[scale=0.55]{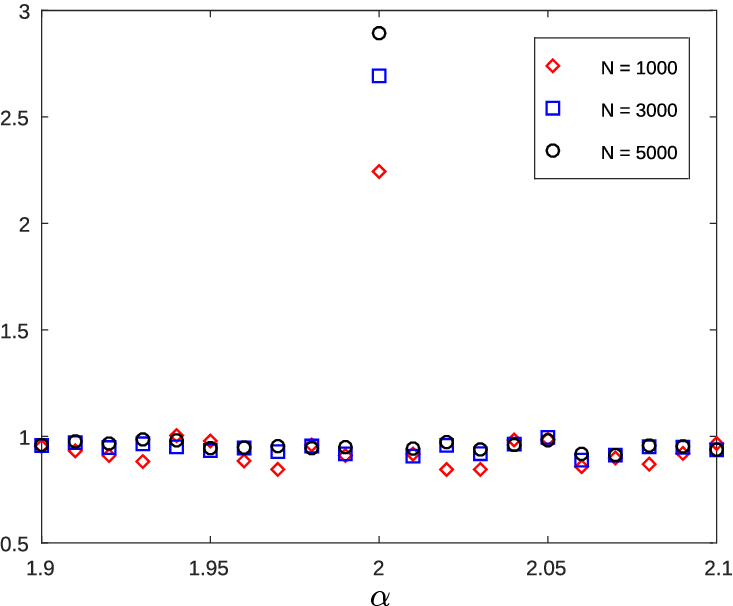}
	\caption{The ordinate of the left (resp.\ right) figure is the value $\cN_2^{(1,2)}(N)/N^{3-\alpha}$ (resp.\ $\cN_2^{(3)}(N)/N^{3-\alpha}$) 
	divided by the right-hand side of \eqref{eq:conj1} (resp.\ \eqref{eq:conj2}), 
	where $N\in\{ 1000+2000i : i=0,1,2 \}$ and $\alpha\in\{ 1.9+0.01i : i=0,1,\ldots,20 \}$.}\label{fig1}
\end{figure}

\begin{figure}
	\centering
	\includegraphics[scale=0.55]{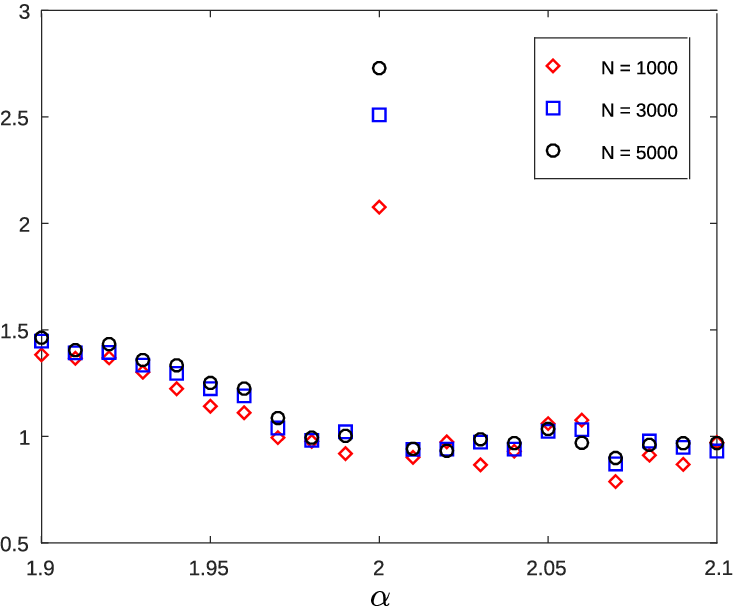}
	\hspace{2em}
	\includegraphics[scale=0.55]{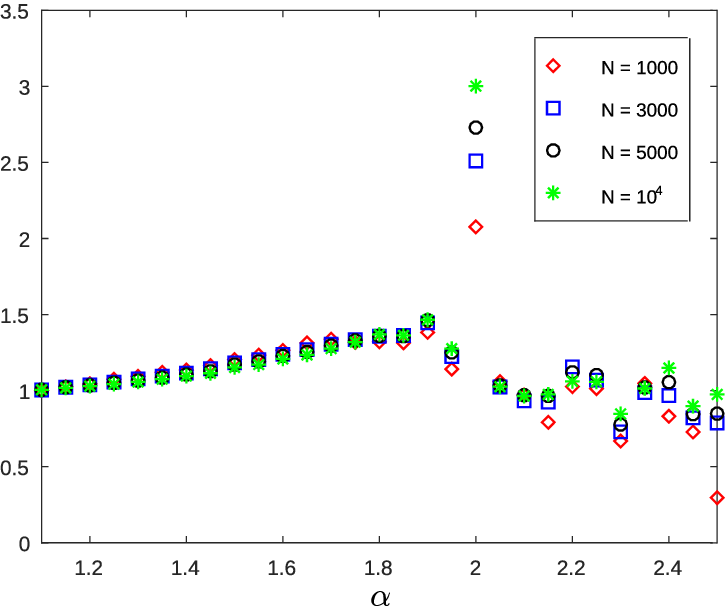}
	\caption{The ordinate of the left (resp.\ right) figure is the value $\cN_\alpha^{\AP}(N)/N^{3-\alpha}$ 
	divided by the right-hand side of \eqref{eq:conj3}, 
	where $N\in\{ 1000+2000i : i=0,1,2 \}$ (resp.\ $N\in\{ 1000+2000i : i=0,1,2 \}\cup\{10^4\}$) and 
	$\alpha\in\{ 1.9+0.01i : i=0,1,\ldots,20 \}$ (resp.\ $\alpha\in\{ 1.1+0.05i : i=0,1,\ldots,28 \}$).}\label{fig2}
\end{figure}

Now, let us consider the following two statements: 
for Lebesgue-a.e.\ $\alpha\in(1,\alpha_0)$ (resp.\ $\alpha>\alpha_0$), 
the equation $x+y=z$ has infinitely (resp.\ at most finitely) many solutions in $\PS(\alpha)$.
Saito \cite{Saito2} conjectured $\alpha_0=3$ for which the two statements hold.
He showed \cite[Theorem~1.2]{Saito2} that, for Lebesgue-a.e.\ $\alpha>3$, 
the equation $x+y=z$ has at most finitely many solutions in $\PS(\alpha)$.
Eqs.~\eqref{eq:conj1} and \eqref{eq:conj2} support his conjecture, 
since $N^{3-\alpha}$ vanishes or diverges to positive infinity as $N\to\infty$ according to $\alpha>3$ or $\alpha<3$.

\begin{proof}[Heuristic argument on Conjecture~$\ref{conj1}$]
	It can easily be checked that 
	\begin{equation*}
		\cN_\alpha^{(1,2)}(N) = \#\left\{ (n_1,n_2)\in\mathbb{N}^2 : 
		\begin{array}{c}
			n_1,n_2\le N^\alpha,\\
			\{-n_1^{1/\alpha}\} < \phi_\alpha(n_1),\ \{-n_2^{1/\alpha}\} < \phi_\alpha(n_2),\\
			\{-(n_1+n_2)^{1/\alpha}\} < \phi_\alpha(n_1+n_2)
		\end{array}
		\right\}
	\end{equation*}
	in the same way as the proof of Theorem~\ref{main1}.
	Assume that the sequences $(-n_1^{1/\alpha})$, $(-n_2^{1/\alpha})$, and $\bigl( -(n_1+n_2)^{1/\alpha} \bigr)$ 
	are ``sufficiently'' equidistributed modulo $1$ when $n_1$ and $n_2$ run over all integers in the interval $[1,N^\alpha]$.
	Then 
	\begin{equation}
	\begin{split}
		\cN_\alpha^{(1,2)}(N) &\sim \sum_{1\le n_1,n_2\le N^\alpha} \phi_\alpha(n_1)\phi_\alpha(n_2)\phi_\alpha(n_1+n_2)\\
		&\sim \sum_{1\le n_1,n_2\le N^\alpha} \alpha^{-3}\bigl( n_1n_2(n_1+n_2) \bigr)^{1/\alpha-1}\\
		&= N^{3-\alpha}\cdot\frac{\alpha^{-3}}{N^{2\alpha}}\sum_{1\le n_1,n_2\le N^\alpha}
		\Bigl( \frac{n_1}{N^\alpha}\cdot\frac{n_2}{N^\alpha}\cdot\frac{n_1+n_2}{N^\alpha} \Bigr)^{1/\alpha-1}\\
		&\sim N^{3-\alpha}\cdot\alpha^{-3}\iint_{0<x,y\le1} \bigl( xy(x+y) \bigr)^{1/\alpha-1}\,dxdy
	\end{split}\label{eq:conj1'}
	\end{equation}
	as $N\to\infty$.
	Also, assume that the sequences $(-n_1^{1/\alpha})$, $(-n_2^{1/\alpha})$, and $\bigl( -(n_1+n_2)^{1/\alpha} \bigr)$ 
	are ``sufficiently'' equidistributed modulo $1$ when $n_1$ and $n_2$ run over all positive integers with $n_1+n_2\le N^\alpha$.
	Then 
	\begin{equation}
	\begin{split}
		\cN_\alpha^{(3)}(N) &= \#\left\{ (n_1,n_2)\in\mathbb{N}^2 : 
		\begin{array}{c}
			n_1+n_2\le N^\alpha,\\
			\{-n_1^{1/\alpha}\} < \phi_\alpha(n_1),\ \{-n_2^{1/\alpha}\} < \phi_\alpha(n_2),\\
			\{-(n_1+n_2)^{1/\alpha}\} < \phi_\alpha(n_1+n_2)
		\end{array}
		\right\}\\
		&\sim N^{3-\alpha}\cdot\alpha^{-3}\iint_{\substack{x,y>0 \\ x+y\le1}} \bigl( xy(x+y) \bigr)^{1/\alpha-1}\,dxdy
	\end{split}\label{eq:conj2'}
	\end{equation}
	as $N\to\infty$.
	By the change of variables $(x,y)=( u(1-v), uv )$, it turns out that 
	\begin{equation}
	\begin{split}
		&\quad \iint_{\substack{x,y>0 \\ x+y\le1}} \bigl( xy(x+y) \bigr)^{1/\alpha-1}\,dxdy
		= \iint_{0<u,v\le1} \bigl( u^3(1-v)v \bigr)^{1/\alpha-1}u\,dudv\\
		&= \int_0^1 u^{3/\alpha-2}\,du \int_0^1 \bigl( (1-v)v \bigr)^{1/\alpha-1}\,dv
		= \frac{B(1/\alpha,1/\alpha)}{3/\alpha-1}
		= \frac{\alpha\Gamma(1/\alpha)^2}{(3-\alpha)\Gamma(2/\alpha)}\\
		&= \frac{\alpha^3\Gamma(1+1/\alpha)^2}{(3-\alpha)\Gamma(2/\alpha)}.
	\end{split}\label{eq:int2}
	\end{equation}
	By \eqref{eq:conj2'} and \eqref{eq:int2}, we obtain \eqref{eq:conj2}.
	Similarly, 
	\begin{align*}
		&\quad \iint_{\substack{0<x,y\le1 \\ x+y>1}} \bigl( xy(x+y) \bigr)^{1/\alpha-1}\,dxdy\\
		&= \int_1^2 u^{3/\alpha-2}\,du \int_{1-1/u}^{1/u} \bigl( (1-v)v \bigr)^{1/\alpha-1}\,dv
		= \alpha^3I(\alpha).
	\end{align*}
	Therefore, 
	\begin{equation}
		\alpha^{-3}\iint_{0<x,y\le1} \bigl( xy(x+y) \bigr)^{1/\alpha-1}\,dxdy
		= \frac{\Gamma(1+1/\alpha)^2}{(3-\alpha)\Gamma(2/\alpha)} + I(\alpha).
		\label{eq:int1}
	\end{equation}
	By \eqref{eq:conj1'} and \eqref{eq:int1}, we obtain \eqref{eq:conj1}.
\end{proof}

\begin{proof}[Heuristic argument on Conjecture~$\ref{conj2}$]
	Assume that the sequences $(-n_1^{1/\alpha})$, $(-n_2^{1/\alpha})$, and $\bigl( -(2n_2-n_1)^{1/\alpha} \bigr)$ 
	are ``sufficiently'' equidistributed modulo $1$ when $n_1$ and $n_2$ run over all positive integers with $n_1<n_2<2n_2-n_1\le N^\alpha$.
	Then 
	\begin{equation}
	\begin{split}
		\cN_\alpha^{\AP}(N) &= \#\left\{ (n_1,n_2)\in\mathbb{N}^2 : 
		\begin{array}{c}
			1\le n_1<n_2<2n_2-n_1\le N^\alpha,\\
			\{-n_1^{1/\alpha}\} < \phi_\alpha(n_1),\ \{-n_2^{1/\alpha}\} < \phi_\alpha(n_2),\\
			\{-(2n_2-n_1)^{1/\alpha}\} < \phi_\alpha(2n_2-n_1)
		\end{array}
		\right\}\\
		&\sim N^{3-\alpha}\cdot\alpha^{-3}\iint_{0<x<y<2y-x\le1} \bigl( xy(2y-x) \bigr)^{1/\alpha-1}\,dxdy.
	\end{split}\label{eq:conj3'}
	\end{equation}
	as $N\to\infty$.
	By the change of variables $(x,z)=(x,2y-x)$, it turns out that 
	\begin{equation}
	\begin{split}
		&\quad \iint_{0<x<y<2y-x\le1} \bigl( xy(2y-x) \bigr)^{1/\alpha-1}\,dxdy\\
		&= 2^{-1/\alpha}\iint_{0<x<z\le1} \bigl( xz(x+z) \bigr)^{1/\alpha-1}\,dxdz\\
		&= 2^{-1/\alpha-1}\iint_{0<x,z\le1} \bigl( xz(x+z) \bigr)^{1/\alpha-1}\,dxdz.
	\end{split}\label{eq:int3}
	\end{equation}
	By \eqref{eq:conj3'}, \eqref{eq:int3}, and \eqref{eq:int1}, 
	we obtain \eqref{eq:conj3}.
\end{proof}

\section*{Acknowledgments}
The author thanks Dr.\ Kota Saito for discussing a heuristic way together in attending a conference.
The author was supported by JSPS KAKENHI Grant Numbers JP22J00339 and JP22KJ1621.

\appendix
\section{Proof of Theorem~$\ref{mainDes2}$}

In this appendix, we prove Theorem~\ref{mainDes2}.
For this purpose, we begin with the following lemma, 
which is a simple generalization of \cite[pp.~394--395, Fundamental lemma]{Deshouillers2}.

\begin{lemma}\label{lemDes2}
	Let $\alpha_1,\alpha_2>1$ be real numbers, and $N\ge2$ be an integer.
	If a positive integer $n_1\le N^{1/\alpha_1}$ satisfies the following two conditions, 
	then there exists a positive integer $n_2$ such that $\lfloor{n_1^{\alpha_1}}\rfloor+\lfloor{n_2^{\alpha_2}}\rfloor=N$.
	\renewcommand{\theenumi}{$\arabic{enumi}$}
	\renewcommand{\labelenumi}{\theenumi.}
	\begin{enumerate}
		\item
		$1 - \{ (N+1/2-n_1^{\alpha_1})^{1/\alpha_2} \} \le (1/2\alpha_2)(N+1/2)^{1/\alpha_2-1}$.
		\item
		$\{ n_1^{\alpha_1} \}<1/2$.
	\end{enumerate}
\end{lemma}
\begin{proof}
	Assume that a positive integer $n_1\le N^{1/\alpha_1}$ satisfies conditions~1--2.
	Set $n_2=\lfloor{(N+1/2-n_1^{\alpha_1})^{1/\alpha_2}}\rfloor+1$.
	By condition~1, 
	\begin{equation}
	\begin{split}
		(N+1/2-n_1^{\alpha_1})^{1/\alpha_2} < n_2
		&= (N+1/2-n_1^{\alpha_1})^{1/\alpha_2} + 1 - \{ (N+1/2-n_1^{\alpha_1})^{1/\alpha_2} \}\\
		&\le (N+1/2-n_1^{\alpha_1})^{1/\alpha_2-1} + (1/2\alpha_2)(N+1/2)^{1/\alpha_2-1}.
	\end{split}\label{eq20}
	\end{equation}
	The mean value theorem implies that 
	\[
	(N+1-n_1^{\alpha_1})^{1/\alpha_2} = (N+1/2-n_1^{\alpha_1})^{1/\alpha_2} + \frac{1}{2\alpha_2}(N+1/2-n_1^{\alpha_1}+\theta)^{1/\alpha_2-1}
	\]
	for some $\theta\in(1/2,1)$.
	By this and $0 < N+1/2-n_1^{\alpha_1}+\theta < N+1/2$, we have 
	\[
	(N+1-n_1^{\alpha_1})^{1/\alpha_2} > (N+1/2-n_1^{\alpha_1})^{1/\alpha_2} + (1/2\alpha_2)(N+1/2)^{1/\alpha_2-1}.
	\]
	This and \eqref{eq20} yield that 
	\[
	(N+1/2-n_1^{\alpha_1})^{1/\alpha_2} < n_2 < (N+1-n_1^{\alpha_1})^{1/\alpha_2},
	\]
	which is equivalent to 
	\[
	N+1/2 < n_1^{\alpha_1}+n_2^{\alpha_2} < N+1.
	\]
	By condition~2, 
	\begin{align*}
		N+1 &> \lfloor{n_1^{\alpha_1}}\rfloor+\lfloor{n_2^{\alpha_2}}\rfloor
		= n_1^{\alpha_1}+n_2^{\alpha_2} - \{ n_1^{\alpha_1} \} - \{ n_2^{\alpha_2} \}\\
		&> (N+1/2)-1/2-1 = N-1.
	\end{align*}
	Therefore, $\lfloor{n_1^{\alpha_1}}\rfloor+\lfloor{n_2^{\alpha_2}}\rfloor=N$.
\end{proof}

An outline of the proof of Theorem~\ref{mainDes2} is as follows.
By Lemmas~\ref{lemDes2} and \ref{Koksma}, 
Theorem~\ref{mainDes2} reduces to estimating exponential sums.
To estimate exponential sums, we use Lemmas~\ref{2ndderiv}, \ref{3rdderiv}, and \ref{exp-pair}.
Although Deshouillers \cite{Deshouillers2} used the third derivative test of van der Corput, 
we use Lemma~\ref{3rdderiv} instead.

\begin{proof}[Proof of Theorem~$\ref{mainDes2}$]
	Without loss of generality, we may assume that $1<\alpha_1\le\alpha_2<3/2$.
	Let $N\ge e^4-1/2$ be an integer.
	Set $c_1=1/2$, $c_2=3/4$, $X=N+1/2$, 
	\begin{equation}
		H_1=\log X \quad\text{and}\quad
		H_2=X^{1-1/\alpha_2}\log X.
		\label{eqH''}
	\end{equation}
	Then both $H_1$ and $H_2$ are greater than or equal to $4$.
	Lemma~\ref{Koksma} implies that 
	the value 
	\begin{align*}
		R &= \#\biggl\{ (c_1X)^{1/\alpha_1}\le n\le(c_2X)^{1/\alpha_1} : 
		\begin{array}{c}
			\{n^{\alpha_1}\}<1/2,\\
			1-\{ (X-n^{\alpha_1})^{1/\alpha_2} \}\le(1/2\alpha_2)X^{1/\alpha_2-1}
		\end{array}
		\biggr\}\\
		&\quad- \frac{c_2^{1/\alpha_1}-c_1^{1/\alpha_1}}{4\alpha_2}X^{1/\alpha_1+1/\alpha_2-1}
	\end{align*}
	satisfies the inequality 
	\[
	R \ll R_0 + R_3 + R_{2,1} + R_{2,2} + R_1, 
	\]
	where $f_1(x)=x^{\alpha_1}$, $f_2(x)=(X-x^{\alpha_1})^{1/\alpha_2}$, $\cI=[(c_1X)^{1/\alpha_1}, (c_2X)^{1/\alpha_1}]$, 
	\begin{equation*}
		R_0 = \abs{(c_2^{1/\alpha_1}-c_1^{1/\alpha_1})X^{1/\alpha_1} - \#(\cI\cap\mathbb{Z})}X^{1/\alpha_2-1},
	\end{equation*}
	\begin{equation*}
		R_3 = \sum_{\substack{1\le\abs{h_1}\le H_1 \\ 1\le\abs{h_2}\le H_2}}
		\abs{h_1}^{-1}\min\{ X^{1/\alpha_2-1}, \abs{h_2}^{-1} \}\abs{\sum_{n\in\cI\cap\mathbb{Z}} e\bigl( h_1f_1(n)+h_2f_2(n) \bigr)},
	\end{equation*}
	\begin{equation*}
		R_{2,1} = X^{1/\alpha_2-1}\sum_{1\le h\le H_1} h^{-1}\abs{\sum_{n\in\cI\cap\mathbb{Z}} e\bigl( hf_1(n) \bigr)},
	\end{equation*}
	\begin{equation*}
		R_{2,2} = \sum_{1\le h\le H_2} \min\{ X^{1/\alpha_2-1}, h^{-1} \}\abs{\sum_{n\in\cI\cap\mathbb{Z}} e\bigl( hf_2(n) \bigr)},
		\quad\text{and}
	\end{equation*}
	\begin{equation*}
		R_1 = X^{1/\alpha_1}\Bigl( \frac{1}{H_2} + \frac{X^{1/\alpha_2-1}}{H_1} \Bigr).
	\end{equation*}
	Also, partition the sum $R_3$ into two sums: 
	\begin{align*}
		R_3 = \sum_{\substack{1\le\abs{h_1}\le H_1 \\ 1\le\abs{h_2}\le H_2}}
		= \sum_{\substack{1\le\abs{h_1}\le H_1 \\ 1\le\abs{h_2}\le H_2 \\ h_1h_2>0}}
		+ \sum_{\substack{1\le\abs{h_1}\le H_1 \\ 1\le\abs{h_2}\le H_2 \\ h_1h_2<0}}
		= R_{3,1}+R_{3,2},\quad\text{say}.
	\end{align*}
	If $R=o(X^{1/\alpha_1+1/\alpha_2-1})$ as $X\to\infty$, 
	then $\cR_{\alpha_1,\alpha_2}(N) \gg N^{1/\alpha_1+1/\alpha_2-1}$ as $N\to\infty$ due to Lemma~\ref{lemDes2}.
	Hence, we show that $R=o(X^{1/\alpha_1+1/\alpha_2-1})$ below.
	\par
	\setcounter{count}{0}
	\textbf{Step~\num.}
	By \eqref{eqH''}, it is clear that 
	\[
	R_0 \le X^{1/\alpha_2-1} \quad\text{and}\quad
	R_1 \ll X^{1/\alpha_1+1/\alpha_2-1}(\log X)^{-1}.
	\]
	Also, $\abs{f'_1(x)} \asymp_{\alpha_1} X^{1-1/\alpha_1}$ when $x\in\cI$.
	Using the exponent pair $(1/2,1/2)$ (see Lemma~\ref{exp-pair}), we have 
	\begin{align*}
		\sum_{n\in\cI\cap\mathbb{Z}} e\bigl( hf_1(n) \bigr)
		&\ll_{\alpha_1} (hX^{1-1/\alpha_1})^{1/2}(X^{1/\alpha_1})^{1/2} + (hX^{1-1/\alpha_1})^{-1}\\
		&= h^{1/2}X^{1/2} + h^{-1}X^{1/\alpha_1-1}
		\ll h^{1/2}X^{1/2}
	\end{align*}
	for every $h\ge1$.
	Thus, 
	\begin{align*}
		\sum_{1\le h\le H_1} h^{-1}\abs{\sum_{n\in\cI\cap\mathbb{Z}} e\bigl( hf_1(n) \bigr)}
		&\ll_{\alpha_1} \sum_{1\le h\le H_1} h^{-1}\cdot h^{1/2}X^{1/2}\\
		&\ll H_1^{1/2}X^{1/2}
		= X^{1/2}(\log X)^{1/2},
	\end{align*}
	where we have used \eqref{eqH''} to obtain the last equality.
	Therefore, 
	\[
	R_{2,1} \ll_{\alpha_1} X^{1/\alpha_2-1/2}(\log X)^{1/2}.
	\]
	\par
	\textbf{Step~\num.}
	The first and second derivatives $f'_2$ and $f''_2$ of $f_2$ are below: 
	\begin{align}
		f'_2(x) &= -\alpha_1\alpha_2^{-1}x^{\alpha_1-1}(X-x^{\alpha_1})^{1/\alpha_2-1},\label{eq:f2'}\\
		\begin{split}
			f''_2(x) &= -\alpha_1\alpha_2^{-1}\bigl( (\alpha_1-1)X+(1-\alpha_1\alpha_2^{-1})x^{\alpha_1} \bigr)
			x^{\alpha_1-2}(X-x^{\alpha_1})^{1/\alpha_2-2}\\
			&= \bigl( (\alpha_1-1)X+(1-\alpha_1\alpha_2^{-1})x^{\alpha_1} \bigr)x^{-1}(X-x^{\alpha_1})^{-1}f'_2(x).
		\end{split}\label{eq:f2''}
	\end{align}
	When $x\in\cI$, we have $\abs{f''_2(x)} \asymp_{\alpha_1,\alpha_2} X^{1/\alpha_2-2/\alpha_1}$.
	This and Lemma~\ref{2ndderiv} imply that 
	\begin{align*}
		\sum_{n\in\cI\cap\mathbb{Z}} e\bigl( hf_2(n) \bigr)
		&\ll_{\alpha_1,\alpha_2} X^{1/\alpha_1}(hX^{1/\alpha_2-2/\alpha_1})^{1/2}+(hX^{1/\alpha_2-2/\alpha_1})^{-1/2}\\
		&= h^{1/2}X^{1/2\alpha_2} + h^{-1/2}X^{1/\alpha_1-1/2\alpha_2}
	\end{align*}
	for every $h\ge1$.
	Therefore, 
	\begin{align*}
		R_{2,2} &\ll_{\alpha_1,\alpha_2} \sum_{1\le h\le H_2} \min\{ X^{1/\alpha_2-1}, h^{-1} \}
		( h^{1/2}X^{1/2\alpha_2} + h^{-1/2}X^{1/\alpha_1-1/2\alpha_2} )\\
		&\le \sum_{1\le h\le H_2} ( h^{-1}\cdot h^{1/2}X^{1/2\alpha_2} + X^{1/\alpha_2-1}\cdot h^{-1/2}X^{1/\alpha_1-1/2\alpha_2} )\\
		&\ll H_2^{1/2}X^{1/2\alpha_2}+H_2^{1/2}X^{1/\alpha_1+1/2\alpha_2-1}\\
		&\ll H_2^{1/2}X^{1/2\alpha_2}
		= X^{1/2}(\log X)^{1/2},
	\end{align*}
	where we have used \eqref{eqH''} to obtain the last equality.
	\par
	\textbf{Step~\num.}
	Let $h_1h_2<0$. From now on, set $f=h_1f_1+h_2f_2$.
	When $x\in\cI$, we have 
	\[
	\abs{f''(x)} \asymp_{\alpha_1,\alpha_2} \abs{h_1}X^{1-2/\alpha_1}+\abs{h_2}X^{1/\alpha_2-2/\alpha_1}.
	\]
	This and Lemma~\ref{2ndderiv} imply that 
	\begin{align*}
		\sum_{n\in\cI\cap\mathbb{Z}} e\bigl( f(n) \bigr)
		&\ll_{\alpha_1,\alpha_2} X^{1/\alpha_1}( \abs{h_1}^{1/2}X^{1/2-1/\alpha_1} + \abs{h_2}^{1/2}X^{1/2\alpha_2-1/\alpha_1} )\\
		&\qquad\qquad+ \min\{ \abs{h_1}^{-1/2}X^{1/\alpha_1-1/2}, \abs{h_2}^{-1/2}X^{1/\alpha_1-1/2\alpha_2} \}\\
		&\le \abs{h_1}^{1/2}X^{1/2} + \abs{h_2}^{1/2}X^{1/2\alpha_2}
		+ \abs{h_2}^{-1/2}X^{1/\alpha_1-1/2\alpha_2}.
	\end{align*}
	Thus, 
	\begin{align*}
		R_{3,2} &\ll_{\alpha_1,\alpha_2} \sum_{\substack{1\le h_1\le H_1 \\ 1\le h_2\le H_2}}
		\abs{h_1}^{-1}\min\{ X^{1/\alpha_2-1}, \abs{h_2}^{-1} \}\\
		&\qquad\qquad\qquad ( \abs{h_1}^{1/2}X^{1/2} + \abs{h_2}^{1/2}X^{1/2\alpha_2}
		+ \abs{h_2}^{-1/2}X^{1/\alpha_1-1/2\alpha_2} ).
	\end{align*}
	By the inequalities 
	\[
	\abs{h_1}^{-1}\min\{ X^{1/\alpha_2-1}, \abs{h_2}^{-1} \}
	\le \abs{h_1}^{-1}\abs{h_2}^{-1},\ \abs{h_1}^{-1}X^{1/\alpha_2-1},
	\]
	it turns out that 
	\begin{align*}
		R_{3,2} &\ll_{\alpha_1,\alpha_2} \sum_{\substack{1\le h_1\le H_1 \\ 1\le h_2\le H_2}}
		\bigl( h_1^{-1}h_2^{-1}(\abs{h_1}^{1/2}X^{1/2} + \abs{h_2}^{1/2}X^{1/2\alpha_2})\\
		&\qquad\qquad\qquad+ h_1^{-1}X^{1/\alpha_2-1}\cdot\abs{h_2}^{-1/2}X^{1/\alpha_1-1/2\alpha_2} \bigr)\\
		&\ll H_1^{1/2}(\log H_2)X^{1/2} + (\log H_1)H_2^{1/2}X^{1/2\alpha_2}\\
		&\quad+ (\log H_1)H_2^{1/2}X^{1/\alpha_1+1/2\alpha_2-1}\\
		&\ll_{\alpha_2} X^{1/2}(\log X)^{3/2} + X^{1/2}(\log X)^{1/2}(\log\log X)\\
		&\quad+ X^{1/\alpha_1-1/2}(\log X)^{1/2}(\log\log X)\\
		&\ll X^{1/2}(\log X)^{3/2},
	\end{align*}
	where we have used \eqref{eqH''} to obtain the second to last inequality.
	\par
	\textbf{Step~\num.}
	Let $h_1h_2>0$. When $x\in\cI$, we have 
	\begin{equation}
		\abs{f''_1(x)} \asymp_{\alpha_1} X^{1-2/\alpha_1} \quad\text{and}\quad
		\abs{f''_2(x)} \asymp_{\alpha_1,\alpha_2} X^{1/\alpha_2-2/\alpha_1}.
		\label{eq:f1''f2''}
	\end{equation}
	Take a small $\delta_2=\delta_2(\alpha_1,\alpha_2)\in(0,1)$ such that 
	if $\abs{h_2}\le\delta_2X^{1-1/\alpha_2}$ and $x\in\cI$, then 
	\begin{equation}
		\abs{f''(x)} = \abs{h_1f''_1(x)}-\abs{h_2f''_2(x)} \asymp_{\alpha_1,\alpha_2} \abs{h_1}X^{1-2/\alpha_1}.
		\label{eq24}
	\end{equation}
	(Note that $h_1$ is a non-zero integer.) 
	If $\abs{h_2}\le\delta_2X^{1-1/\alpha_2}$, Lemma~\ref{2ndderiv} and \eqref{eq24} imply that 
	\begin{align*}
		\sum_{n\in\cI\cap\mathbb{Z}} e\bigl( f(n) \bigr)
		&\ll_{\alpha_1,\alpha_2} X^{1/\alpha_1}(\abs{h_1}X^{1-2/\alpha_1})^{1/2} + (\abs{h_1}X^{1-2/\alpha_1})^{-1/2}\\
		&= \abs{h_1}^{1/2}X^{1/2} + \abs{h_1}^{-1/2}X^{1/\alpha_1-1/2}
		\ll \abs{h_1}^{1/2}X^{1/2}.
	\end{align*}
	This yields that 
	\begin{align*}
		R_{3,1,1} &\coloneqq \sum_{\substack{1\le\abs{h_1}\le H_1 \\ 1\le\abs{h_2}\le\delta_2X^{1-1/\alpha_2} \\ h_1h_2>0}}
		\abs{h_1}^{-1}\min\{ X^{1/\alpha_2-1}, \abs{h_2}^{-1} \}\abs{\sum_{n\in\cI\cap\mathbb{Z}} e\bigl( f(n) \bigr)}\\
		&\ll_{\alpha_1,\alpha_2} \sum_{\substack{1\le h_1\le H_1 \\ 1\le h_2\le\delta_2X^{1-1/\alpha_2}}}
		h_1^{-1}X^{1/\alpha_2-1}\cdot\abs{h_1}^{1/2}X^{1/2}\\
		&\ll H_1^{1/2}\cdot\delta_2X^{1-1/\alpha_2}\cdot X^{1/\alpha_2-1/2}
		\le X^{1/2}(\log X)^{1/2},
	\end{align*}
	where we have used \eqref{eqH''} and $\delta_2\in(0,1)$ to obtain the last inequality.
	\par
	\textbf{Step~\num.}
	Now, we investigate the third derivative $f'''$ of $f$ to estimate 
	\[
	R_{3,1,2} \coloneqq \sum_{\substack{1\le\abs{h_1}\le H_1 \\ \delta_2X^{1-1/\alpha_2}<\abs{h_2}\le H_2 \\ h_1h_2>0}}
	\abs{h_1}^{-1}\min\{ X^{1/\alpha_2-1}, \abs{h_2}^{-1} \}\abs{\sum_{n\in\cI\cap\mathbb{Z}} e\bigl( f(n) \bigr)}.
	\]
	Let $x\in\cI$. We show that 
	(i) $f'''(x)<0$ if $h_1,h_2>0$; 
	(ii) $f'''(x)>0$ if $h_1,h_2<0$.
	By \eqref{eq:f2''}, the third derivative $f'''_2$ of $f_2$ is 
	\begin{equation}
	\begin{split}
		f'''_2(x) &= \bigl( (\alpha_1-1)X+(1-\alpha_1\alpha_2^{-1})x^{\alpha_1} \bigr)x^{-1}(X-x^{\alpha_1})^{-1}f''_2(x)\\
		&\quad+ (1-\alpha_1\alpha_2^{-1})\alpha_1x^{\alpha_1-1}\cdot x^{-1}(X-x^{\alpha_1})^{-1}f'_2(x)\\
		&\quad+ \bigl( (\alpha_1-1)X+(1-\alpha_1\alpha_2^{-1})x^{\alpha_1} \bigr)(-x^{-2})(X-x^{\alpha_1})^{-1}f'_2(x)\\
		&\quad+ \bigl( (\alpha_1-1)X+(1-\alpha_1\alpha_2^{-1})x^{\alpha_1} \bigr)x^{-1}\cdot\alpha_1x^{\alpha_1-1}(X-x^{\alpha_1})^{-2}f'_2(x).
	\end{split}\label{eq21}
	\end{equation}
	By \eqref{eq:f2''}, the third and fourth lines of the right-hand side of \eqref{eq21} are equal to 
	\[
	- x^{-1}f''_2(x) \quad\text{and}\quad \alpha_1x^{\alpha_1-1}(X-x^{\alpha_1})^{-1}f''_2(x),
	\]
	respectively. Thus, the sum of the first, third, and fourth lines of the right-hand side of \eqref{eq21} is equal to 
	\begin{align*}
		&\quad \bigl( (\alpha_1-1)X+(1-\alpha_1\alpha_2^{-1})x^{\alpha_1} \bigr)x^{-1}(X-x^{\alpha_1})^{-1}f''_2(x)\\
		&\quad- x^{-1}f''_2(x) + \alpha_1x^{\alpha_1-1}(X-x^{\alpha_1})^{-1}f''_2(x)\\
		&= \bigl( (\alpha_1-2)X+(2+\alpha_1-\alpha_1\alpha_2^{-1})x^{\alpha_1} \bigr)x^{-1}(X-x^{\alpha_1})^{-1}f''_2(x).
	\end{align*}
	By this and \eqref{eq21}, 
	\begin{equation}
	\begin{split}
		f'''_2(x) &= \bigl( (\alpha_1-2)X+(2+\alpha_1-\alpha_1\alpha_2^{-1})x^{\alpha_1} \bigr)x^{-1}(X-x^{\alpha_1})^{-1}f''_2(x)\\
		&\quad+ (1-\alpha_1\alpha_2^{-1})\alpha_1x^{\alpha_1-2}(X-x^{\alpha_1})^{-1}f'_2(x).
	\end{split}\label{eq:f2'''}
	\end{equation}
	By $c_1=1/2$, we have 
	\begin{equation}
	\begin{split}
		&\quad (\alpha_1-2)X+(2+\alpha_1-\alpha_1\alpha_2^{-1})x^{\alpha_1}\\
		&\ge (\alpha_1-2)X+(2+\alpha_1-\alpha_1\alpha_2^{-1})c_1X\\
		&= (-2+3\alpha_1-\alpha_1\alpha_2^{-1})X/2\\
		&\ge (3\alpha_1-3)X/2 > 0.
	\end{split}\label{eq22}
	\end{equation}
	Also, by \eqref{eq:f2'} and \eqref{eq:f2''}, both $f'_2(x)$ and $f''_2(x)$ are negative.
	From this, \eqref{eq:f2'''}, and \eqref{eq22}, it follows that $f'''_2(x)<0$.
	Thus, if $h_1,h_2>0$ (resp.\ $h_1,h_2<0$), then both $h_1f'''_1(x)$ and $h_2f'''_2(x)$ are negative (resp.\ positive), 
	which implies that $f'''(x)<0$ (resp.\ $f'''(x)>0$).
	\par
	\textbf{Step~\num.}
	We calculate and estimate $f'''(x) - (\alpha_1-2)x^{-1}f''(x)$.
	By \eqref{eq:f2''} and \eqref{eq:f2'''}, 
	\begin{align*}
		&\quad f'''_2(x) - (\alpha_1-2)x^{-1}f''_2(x)\\
		&= (2\alpha_1-\alpha_1\alpha_2^{-1})x^{\alpha_1}x^{-1}(X-x^{\alpha_1})^{-1}f''_2(x)\\
		&\quad+ (1-\alpha_1\alpha_2^{-1})\alpha_1x^{\alpha_1-2}(X-x^{\alpha_1})^{-1}f'_2(x)\\
		&= \Bigl( (2\alpha_1-\alpha_1\alpha_2^{-1})x^{\alpha_1-1}(X-x^{\alpha_1})^{-1}\\
		&\qquad+ \frac{(1-\alpha_1\alpha_2^{-1})\alpha_1x^{\alpha_1-1}}{(\alpha_1-1)X+(1-\alpha_1\alpha_2^{-1})x^{\alpha_1}} \Bigr)f''_2(x).
	\end{align*}
	Thus, 
	\begin{equation}
	\begin{split}
		&\quad f'''(x) - (\alpha_1-2)x^{-1}f''(x)
		= h_2\bigl( f'''_2(x) - (\alpha_1-2)x^{-1}f''_2(x) \bigr)\\
		&= \Bigl( (2\alpha_1-\alpha_1\alpha_2^{-1})x^{\alpha_1-1}(X-x^{\alpha_1})^{-1}\\
		&\qquad+ \frac{(1-\alpha_1\alpha_2^{-1})\alpha_1x^{\alpha_1-1}}{(\alpha_1-1)X+(1-\alpha_1\alpha_2^{-1})x^{\alpha_1}} \Bigr)h_2f''_2(x).
	\end{split}\label{eq23}
	\end{equation}
	By \eqref{eq23} and \eqref{eq:f1''f2''}, 
	\begin{equation}
		\abs{f'''(x) - (\alpha_1-2)x^{-1}f''(x)}
		\asymp_{\alpha_1,\alpha_2} \abs{h_2}X^{1/\alpha_2-3/\alpha_1}
		\label{eq25}
	\end{equation}
	for every $x\in\cI$.
	\par
	\textbf{Step~\num.}
	Let $h_1h_2>0$ and $\delta_2X^{1-1/\alpha_2}<\abs{h_2}\le H_2$.
	By \eqref{eq25}, we can take a small $\delta_0=\delta_0(\alpha_1,\alpha_2,\delta_2)\in(0,1)$ such that 
	if $x\in\cI$ and $\abs{f''(x)}\le\delta_0X^{1-2/\alpha_1}$, then 
	\begin{align*}
		\abs{f'''(x)} &\ge \abs{f'''(x) - (\alpha_1-2)x^{-1}f''(x)}-\abs{(\alpha_1-2)x^{-1}f''(x)}\\
		&\gg_{\alpha_1,\alpha_2,\delta_2,\delta_0} \abs{h_2}X^{1/\alpha_2-3/\alpha_1}.
	\end{align*}
	Also, if $x\in\cI$ and $\abs{f''(x)}\le\delta_0X^{1-2/\alpha_1}$, then 
	\begin{align*}
		\abs{f'''(x)} &\le \abs{f'''(x) - (\alpha_1-2)x^{-1}f''(x)}+\abs{(\alpha_1-2)x^{-1}f''(x)}\\
		&\ll_{\alpha_1,\alpha_2} \abs{h_2}X^{1/\alpha_2-3/\alpha_1}+X^{1-3/\alpha_1}
		\ll_{\delta_2} \abs{h_2}X^{1/\alpha_2-3/\alpha_1}
	\end{align*}
	due to \eqref{eq25} and $\delta_0\in(0,1)$.
	\par
	\textbf{Step~\num.}
	By $\alpha_1,\alpha_2\in(1,2)$ and $1/\alpha_1+6/\alpha_2>13/3$, 
	the inequality 
	\[
	\max\{ 0, 7/6-1/2\alpha_1-1/\alpha_2 \} < \min\{ 1/\alpha_1, 2/\alpha_2-1 \}
	\]
	holds. Indeed, the inequality $7/6-1/2\alpha_1-1/\alpha_2 < 1/\alpha_1$ follows from 
	\[
	3/\alpha_1+2/\alpha_2
	= (1/3)(1/\alpha_1+6/\alpha_2) + 8/3\alpha_1
	> 13/9 + 4/3 = 25/3 > 7/3;
	\]
	the inequality $7/6-1/2\alpha_1-1/\alpha_2 < 2/\alpha_2-1$ follows from $1/\alpha_1+6/\alpha_2 > 13/3$; 
	the other inequalities are trivial.
	Also, take a real number $c_0$ with 
	\begin{equation}
		\max\{ 0, 7/6-1/2\alpha_1-1/\alpha_2 \} < c_0 < \min\{ 1/\alpha_1, 2/\alpha_2-1 \},
		\label{eq:c0}
	\end{equation}
	and an integer $m\ge2$ with 
	\begin{equation}
		1/\alpha_1+1/\alpha_2 > 3/2+2\cdot4^{-m}c_0.
		\label{eq:m}
	\end{equation}
	\par
	\textbf{Step~\num.}
	Let $h_1h_2>0$ and $\delta_2X^{1-1/\alpha_2}<\abs{h_2}\le H_2$.
	Define the positive numbers $\gamma_1<\cdots<\gamma_m<1/\alpha_1$ as $\gamma_k=4^{k-m}c_0$.
	Partition the interval $\cI$ into the following subsets $\cI_1,\ldots,\cI_{m+1}$: 
	\begin{align*}
		\cI_1 &= \{ x\in\cI : \abs{f''(x)} > \delta_0X^{1-2/\alpha_1-\gamma_1} \},\\
		\cI_k &= \{ x\in\cI : \delta_0X^{1-2/\alpha_1-\gamma_k} < \abs{f''(x)} \le \delta_0X^{1-2/\alpha_1-\gamma_{k-1}} \} \quad(2\le k\le m),\\
		\cI_{m+1} &= \{ x\in\cI : \abs{f''(x)} \le \delta_0X^{1-2/\alpha_1-\gamma_m} \}.
	\end{align*}
	Since $f''$ is monotone due to step~5, 
	for $k=1,\ldots,m+1$, the set $\cI_k$ is a union of at most two intervals.
	Moreover, for every integer $2\le k\le m+1$, 
	\begin{equation}
		\abs{\cI_k} \ll_{\alpha_1,\alpha_2,\delta_2,\delta_0} \abs{h_2}^{-1}X^{1+1/\alpha_1-1/\alpha_2-\gamma_{k-1}}.
		\label{eq:|Ik|}
	\end{equation}
	Indeed, if $2\le k\le m+1$ and $x_1,x_2\in\cI_k$, then the mean value theorem and step~7 imply that 
	\begin{align*}
		2\delta_0X^{1-2/\alpha_1-\gamma_{k-1}}
		&\ge \abs{f''(x_2)-f''(x_1)}
		= \abs{f'''(x_3)(x_2-x_1)}\\
		&\gg_{\alpha_1,\alpha_2,\delta_2,\delta_0} \abs{h_2}X^{1/\alpha_2-3/\alpha_1}\abs{x_2-x_1}
	\end{align*}
	for some $x_3$ between $x_1$ and $x_2$ inclusive; therefore, 
	\[
	\abs{x_2-x_1} \ll_{\alpha_1,\alpha_2,\delta_2,\delta_0} \abs{h_2}^{-1}X^{1+1/\alpha_1-1/\alpha_2-\gamma_{k-1}}.
	\]
	\par
	\textbf{Step~\num.}
	Let $h_1h_2>0$ and $\delta_2X^{1-1/\alpha_2}<\abs{h_2}\le H_2$.
	Define the sums $R_{3,1,2,k}$, $k=1,\ldots,m+1$, as 
	\[
	R_{3,1,2,k} = \sum_{\substack{1\le\abs{h_1}\le H_1 \\ \delta_2X^{1-1/\alpha_2}<\abs{h_2}\le H_2 \\ h_1h_2>0}}
	\abs{h_1}^{-1}\min\{ X^{1/\alpha_2-1}, \abs{h_2}^{-1} \}\abs{\sum_{n\in\cI_k\cap\mathbb{Z}} e\bigl( f(n) \bigr)}.
	\]
	Then 
	\begin{equation}
		R_{3,1,2} \le \sum_{k=1}^{m+1} R_{3,1,2,k}.
		\label{eq26}
	\end{equation}
	First, we estimate $R_{3,1,2,m+1}$.
	By step~7, 
	\begin{equation}
		\abs{f'''(x)} \asymp_{\alpha_1,\alpha_2,\delta_2,\delta_0} \abs{h_2}X^{1/\alpha_2-3/\alpha_1}
		\label{eq:f'''}
	\end{equation}
	for every $x\in\cI_{m+1}$.
	Lemma~\ref{3rdderiv}, \eqref{eq:f'''}, and \eqref{eq:|Ik|} imply that 
	\begin{align*}
		\sum_{n\in\cI_{m+1}\cap\mathbb{Z}} e\bigl( f(n) \bigr)
		&\ll_{\alpha_1,\alpha_2,\delta_2,\delta_0}
		\abs{h_2}^{-1}X^{1+1/\alpha_1-1/\alpha_2-\gamma_m}(\abs{h_2}X^{1/\alpha_2-3/\alpha_1})^{1/6}\\
		&\qquad\qquad\qquad+ (\abs{h_2}X^{1/\alpha_2-3/\alpha_1})^{-1/3}\\
		&= \abs{h_2}^{-5/6}X^{1+1/2\alpha_1-5/6\alpha_2-\gamma_m} + \abs{h_2}^{-1/3}X^{1/\alpha_1-1/3\alpha_2}.
	\end{align*}
	Thus, 
	\begin{align*}
		R_{3,1,2,m+1} &\ll_{\alpha_1,\alpha_2,\delta_2,\delta_0} \sum_{\substack{1\le h_1\le H_1 \\ \delta_2X^{1-1/\alpha_2}<h_2\le H_2}}
		h_1^{-1}h_2^{-1}( \abs{h_2}^{-5/6}X^{1+1/2\alpha_1-5/6\alpha_2-\gamma_m}\\
		&\qquad\qquad\qquad\qquad\qquad\qquad\qquad\qquad+ \abs{h_2}^{-1/3}X^{1/\alpha_1-1/3\alpha_2} )\\
		&\ll_{\delta_2} (\log H_1)\Bigl( (X^{1-1/\alpha_2})^{-5/6}X^{1+1/2\alpha_1-5/6\alpha_2-\gamma_m}\\
		&\qquad\qquad\qquad\qquad\qquad+ (X^{1-1/\alpha_2})^{-1/3}X^{1/\alpha_1-1/3\alpha_2} \Bigr)\\
		&= ( X^{1/6+1/2\alpha_1-\gamma_m} + X^{1/\alpha_1-1/3} )\log\log X,
	\end{align*}
	where we have used \eqref{eqH''} to obtain the last equality.
	By $\gamma_m=c_0$, $c_0>7/6-1/2\alpha_1-1/\alpha_2$ of \eqref{eq:c0}, and $\alpha_2\in(1,3/2)$, 
	it turns out that $R_{3,1,2,m+1}=o(X^{1/\alpha_1+1/\alpha_2-1})$ as $X\to\infty$.
	\par
	Next, we estimate $R_{3,1,2,1}$. By \eqref{eq:f1''f2''}, 
	\begin{equation}
		\delta_0X^{1-2/\alpha_1-\gamma_1} \le \abs{f''(x)} \ll_{\alpha_1,\alpha_2} X^{1-2/\alpha_2}\log X
		\label{eq27}
	\end{equation}
	for every $x\in\cI_1$.
	Lemma~\ref{2ndderiv}, \eqref{eq27}, and $\abs{\cI_1}\le\abs{\cI}\ll X^{1/\alpha_1}$ imply that 
	\begin{align*}
		\sum_{n\in\cI_1\cap\mathbb{Z}} e\bigl( f(n) \bigr)
		&\ll_{\alpha_1,\alpha_2,\delta_0} X^{1/\alpha_1}\cdot\frac{X^{1-2/\alpha_1}\log X}{X^{1/2-1/\alpha_1-\gamma_1/2}}
		+ X^{1/\alpha_1-1/2+\gamma_1/2}\\
		&= X^{1/2+\gamma_1/2}\log X + X^{1/\alpha_1-1/2+\gamma_1/2}.
	\end{align*}
	Thus, 
	\begin{align*}
		R_{3,1,2,1} &\ll_{\alpha_1,\alpha_2,\delta_0} \sum_{\substack{1\le h_1\le H_1 \\ \delta_2X^{1-1/\alpha_2}<h_2\le H_2}}
		h_1^{-1}h_2^{-1}( X^{1/2+\gamma_1/2}\log X + X^{1/\alpha_1-1/2+\gamma_1/2} )\\
		&\ll (\log H_1)(\log H_2)( X^{1/2+\gamma_1/2}\log X + X^{1/\alpha_1-1/2+\gamma_1/2} )\\
		&\ll_{\alpha_2} X^{1/2+\gamma_1/2}(\log X)^2\log\log X.
	\end{align*}
	By $\gamma_1=4^{1-m}c_0$ and \eqref{eq:m}, 
	it turns out that $R_{3,1,2,1}=o(X^{1/\alpha_1+1/\alpha_2-1})$ as $X\to\infty$.
	\par
	Finally, we estimate $R_{3,1,2,k}$ for $k=2,3,\ldots,m$.
	Let $2\le k\le m$ be an integer.
	Lemma~\ref{2ndderiv} and \eqref{eq:|Ik|} imply that 
	\begin{align*}
		\sum_{n\in\cI_k\cap\mathbb{Z}} e\bigl( f(n) \bigr)
		&\ll_{\alpha_1,\alpha_2,\delta_2,\delta_0}
		\abs{h_2}^{-1}X^{1+1/\alpha_1-1/\alpha_2-\gamma_{k-1}}\cdot\frac{X^{1-2/\alpha_1-\gamma_{k-1}}}{X^{1/2-1/\alpha_1-\gamma_k/2}}\\
		&\qquad\qquad\qquad+ X^{1/\alpha_1-1/2+\gamma_k/2}\\
		&= \abs{h_2}^{-1}X^{3/2-1/\alpha_2} + X^{1/\alpha_1-1/2+\gamma_k/2},
	\end{align*}
	where we have used $\gamma_k=4\gamma_{k-1}$ to obtain the last equality.
	Thus, 
	\begin{align*}
		R_{3,1,2,k} &\ll_{\alpha_1,\alpha_2,\delta_2,\delta_0} \sum_{\substack{1\le h_1\le H_1 \\ \delta_2X^{1-1/\alpha_2}<h_2\le H_2}}
		h_1^{-1}h_2^{-1}( \abs{h_2}^{-1}X^{3/2-1/\alpha_2} + X^{1/\alpha_1-1/2+\gamma_k/2} )\\
		&\ll_{\delta_2} (\log H_1)\bigl( (X^{1-1/\alpha_2})^{-1}X^{3/2-1/\alpha_2} + (\log H_2)X^{1/\alpha_1-1/2+\gamma_k/2} \bigr)\\
		&\ll_{\alpha_2} ( X^{1/2} + X^{1/\alpha_1-1/2+\gamma_k/2}\log X )\log\log X.
	\end{align*}
	By $\gamma_1<\cdots<\gamma_m=c_0$, $c_0<2/\alpha_2-1$ of \eqref{eq:c0}, and $1/\alpha_1+1/\alpha_2>3/2$, 
	it turns out that $R_{3,1,2,k}=o(X^{1/\alpha_1+1/\alpha_2-1})$ as $X\to\infty$.
	\par
	\textbf{Step~\num.}
	By steps~1--4, all of the values $R_0$, $R_1$, $R_{2,1}$, $R_{2,2}$, $R_{3,2}$, and $R_{3,1,1}$ are 
	$o(X^{1/\alpha_1+1/\alpha_2-1})$ as $X\to\infty$.
	By step~10, for every integer $1\le k\le m+1$, the value $R_{3,1,2,k}$ is $o(X^{1/\alpha_1+1/\alpha_2-1})$ as $X\to\infty$.
	By \eqref{eq26} and $R_{3,1}\le R_{3,1,1}+R_{3,1,2}$, 
	the value $R_{3,1}$ is also $o(X^{1/\alpha_1+1/\alpha_2-1})$ as $X\to\infty$.
	Therefore, $R=o(X^{1/\alpha_1+1/\alpha_2-1})$ as $X\to\infty$.
\end{proof}

\end{document}